\documentclass[a4paper,oneside,english]{amsart}
\usepackage[T1]{fontenc}
\usepackage[latin9]{inputenc}
\usepackage{mathrsfs}
\usepackage{amstext}
\usepackage{amsthm}
\usepackage{amssymb}
\usepackage{stmaryrd}
\usepackage[all]{xy}

\makeatletter

\pdfpageheight\paperheight
\pdfpagewidth\paperwidth

\newcommand{\noun}[1]{\textsc{#1}}

\numberwithin{equation}{section}
\numberwithin{figure}{section}
\theoremstyle{plain}
\newtheorem{thm}{\protect\theoremname}
  \theoremstyle{definition}
  \newtheorem{defn}[thm]{\protect\definitionname}
  \theoremstyle{remark}
  \newtheorem{rem}[thm]{\protect\remarkname}
  \theoremstyle{plain}
  \newtheorem{prop}[thm]{\protect\propositionname}
  \theoremstyle{plain}
  \newtheorem{cor}[thm]{\protect\corollaryname}
  \theoremstyle{plain}
  \newtheorem{lem}[thm]{\protect\lemmaname}

\usepackage[dvipsnames,svgnames,x11names,hyperref]{xcolor}
\usepackage[colorlinks=true,linkcolor=Dark Red, citecolor=Dark Green,linktoc=all]{hyperref}
\usepackage{lmodern}
\renewcommand*{\epsilon}{\varepsilon}
\usepackage{amssymb}
\usepackage{amsfonts}
\usepackage{egothic}
\usepackage[T1]{fontenc}

\usepackage{url}
\usepackage{amsthm}
\usepackage{aliascnt}
\usepackage{lipsum}
\usepackage{mathrsfs}
\usepackage{esint}
\usepackage{url}
\usepackage{amsmath}
\usepackage{dsfont}
\usepackage{bbm}
\usepackage{pdflscape}
\usepackage{bbm}
\usepackage{bussproofs}

\usepackage{tikz}
\tikzset{node distance=2.5cm, auto}
\usetikzlibrary{positioning}

\newcommand{\myar}[2]{\ar^-{#1}[#2]}
\newcommand{\myard}[2]{\ar_-{#1}[#2]}

\SelectTips{cm}{10}

\makeatletter
\def\matrixobject@{%
  \edef \next@{={\DirectionfromtheDirection@ }}%
  \expandafter \toks@ \next@ \plainxy@
  \let\xy@@ix@=\xyq@@toksix@
  \xyFN@ \OBJECT@}
\let\xy@entry@@norm=\entry@@norm
\def\entry@@norm@patched{%
  \let\object@=\matrixobject@
  \xy@entry@@norm }
\AtBeginDocument{\let\entry@@norm\entry@@norm@patched}
\makeatother

\newcommand{\twocong}[2][0.5]{\ar@{}[#2] \save ?(#1)*{\cong}\restore}
\newcommand{\twoeq}[2][0.5]{\ar@{}[#2] \save ?(#1)*{=}\restore}
\newcommand{\ltwocell}[3][0.5]{\ar@{}[#2] \ar@{=>}?(#1)+/r 0.2cm/;?(#1)+/l 0.2cm/^{#3}}
\newcommand{\rtwocell}[3][0.5]{\ar@{}[#2] \ar@{=>}?(#1)+/l 0.2cm/;?(#1)+/r 0.2cm/^{#3}}
\newcommand{\utwocell}[3][0.5]{\ar@{}[#2] \ar@{=>}?(#1)+/d  0.2cm/;?(#1)+/u 0.2cm/_{#3}}
\newcommand{\dtwocell}[3][0.5]{\ar@{}[#2] \ar@{=>}?(#1)+/u  0.2cm/;?(#1)+/d 0.2cm/^{#3}}
\newcommand{\ultwocell}[3][0.5]{\ar@{}[#2] \ar@{=>}?(#1)+/dr  0.2cm/;?(#1)+/ul 0.2cm/^{#3}}
\newcommand{\urtwocell}[3][0.5]{\ar@{}[#2] \ar@{=>}?(#1)+/dl  0.2cm/;?(#1)+/ur 0.2cm/^{#3}}
\newcommand{\dltwocell}[3][0.5]{\ar@{}[#2] \ar@{=>}?(#1)+/ur  0.2cm/;?(#1)+/dl 0.2cm/^{#3}}
\newcommand{\drtwocell}[3][0.5]{\ar@{}[#2] \ar@{=>}?(#1)+/ul  0.2cm/;?(#1)+/dr 0.2cm/^{#3}}

\makeatother

\usepackage{babel}
  \providecommand{\corollaryname}{Corollary}
  \providecommand{\definitionname}{Definition}
  \providecommand{\lemmaname}{Lemma}
  \providecommand{\propositionname}{Proposition}
  \providecommand{\remarkname}{Remark}
\providecommand{\theoremname}{Theorem}

\begin{document}

\title{Generic Bicategories }

\author{Charles Walker}



\address{Department of Mathematics, Macquarie University, NSW 2109, Australia}

\email{charles.walker1@mq.edu.au}

\thanks{The author acknowledges the support of an Australian Government Research
Training Program Scholarship.}

\date{\today}
\begin{abstract}
It is well known that to give an oplax functor of bicategories $\mathbf{1}\to\mathscr{C}$
is to give a comonad in $\mathscr{C}$. Here we generalize this fact,
replacing the terminal bicategory by any bicategory $\mathscr{A}$
for which the composition functor admits generic factorisations. We
call bicategories with this property\emph{ generic}, and show that
for generic bicategories $\mathscr{A}$ one may express the data of
an oplax functor $\mathscr{A}\to\mathscr{C}$ much like the data of
a comonad; the main advantage of this description being that it does
not directly involve composition in $\mathscr{A}$.

We then go on to apply this result to some well known bicategories,
such as cartesian monoidal categories (seen as one object bicategories),
bicategories of spans, and bicategories of polynomials with cartesian
2-cells.
\end{abstract}

\maketitle
\tableofcontents{}

\section{Introduction}

A classical and simple fact about monads in a bicategory $\mathscr{C}$
is that they are in bijection with lax functors $L\colon\mathbf{1}\to\mathscr{C}$
where $\mathbf{1}$ is the terminal bicategory \cite{StreetFTM}.
Dually, comonads in $\mathscr{C}$ correspond to oplax functors $L\colon\mathbf{1}\to\mathscr{C}$.
The purpose of this paper is to provide a generalization of this dual,
showing that this correspondence may be realized as a special case
of a more general result. 

This is done by replacing the terminal bicategory with bicategories
$\mathscr{A}$ satisfying the following special property: every functor
\[
\mathscr{A}_{X,Z}\left(c,-\circ-\right)\colon\mathscr{A}_{Y,Z}\times\mathscr{A}_{X,Y}\to\mathbf{Set},\qquad X,Y,Z,c\in\mathscr{A}
\]
is a coproduct of representables. A more informative and equivalent
characterization is as follows: every composition functor
\[
\circ\colon\mathscr{A}_{Y,Z}\times\mathscr{A}_{X,Y}\to\mathscr{A}_{X,Z},\qquad X,Y,Z\in\mathscr{A}
\]
admits generic factorisations. We will call bicategories \emph{$\mathscr{A}$
}satisfying this property \emph{generic}.

Informally, this property means that the bicategory $\mathscr{A}$
contains ``diagonal'' 2-cells. A simple example of this is given
by taking $\mathscr{A}$ to be a cartesian monoidal category $\left(\mathcal{E},\times,\mathbf{1}\right)$
seen as a one-object bicategory, where we have diagonal maps $\delta\colon T\to T\times T$
for each $T\in\mathcal{E}$. Another example is given by taking $\mathscr{A}$
to be the bicategory of spans $\mathbf{Span}\left(\mathcal{E}\right)$
in a category $\mathcal{E}$ with pullbacks; here our diagonal maps
are morphisms $\delta$ induced into pullbacks as in
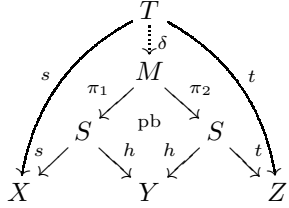
\[
\xymatrix@=1em{ &  & T\ar@{..>}[d]^{\delta}\ar@/^{1pc}/[rrddd]^{t}\ar@/_{1pc}/[llddd]_{s}\\
 &  & M\ar[rd]^{\pi_{2}}\ar[ld]_{\pi_{1}}\ar@{}[dd]|-{\underset{\;}{\textnormal{pb}}}\\
 & S\ar[rd]^{h}\ar[ld]_{s} &  & S\ar[rd]^{t}\ar[ld]_{h}\\
X &  & Y &  & Z
}
\]
such that $\pi_{1}\delta$ and $\pi_{2}\delta$ are identities. This
can also be done for the bicategory of polynomials $\mathbf{Poly}_{c}\left(\mathcal{E}\right)$
with cartesian 2-cells, but becomes more complicated.

Such bicategories also contain ``nullary diagonals'' or augmentations;
these are the 2-cells into identity 1-cells, and turn out to be unique
in such bicategories.

The main result of this paper is that for generic bicategories $\mathscr{A}$,
the functors $\mathscr{A}\to\mathscr{C}$ which respect these diagonals
are precisely the oplax functors. Here ``respecting diagonals''
means that each diagonal $\delta$ and augmentation $\epsilon$ in
$\mathscr{A}$ has a corresponding comultiplication map $\Phi_{\delta}$
and counit map $\Lambda_{\epsilon}$ in $\mathscr{C}$ satisfying
coherence conditions much like those for a comonad.

When the domain bicategory $\mathscr{A}$ is generic, this description
has an important advantage over the usual definition of an oplax functor:
it does not involve composition in the domain bicategory. This reduction
being possible since the information concerning composition in $\mathscr{A}$
is encoded into these diagonal maps. Of course, this property is particularly
useful if composition in $\mathscr{A}$ is complicated; the bicategory
of polynomials being an archetypal example. 

In Section \ref{main} we develop the theory of such bicategories
$\mathscr{A}$ and their diagonal maps, and prove the main result
of this paper, Theorem \ref{generalcoherence}, in which we prove
the equivalence of oplax functors and functors which respect these
diagonals.

In Section \ref{consequences}, we use this result to give a description
of oplax functors out of the bicategory of spans which does not involve
composition of spans (pullbacks), and then give a description of oplax
functors out of the bicategory of polynomials which does not involve
composition of polynomials. 

These descriptions allow for a simpler proof of the universal properties
of spans \cite{unispans}, and a much simpler proof of the universal
properties of polynomials. In our next paper we will use these descriptions
to give an efficient proof of these universal properties.

In Section \ref{GenericDocYoneda} we discuss how this description
of oplax functors can be seen as an instance of doctrinal Yoneda structures,
seen as a consequence of the simpler Day convolution structure on
generic bicategories.

\section{Properties of generic bicategories\label{main}}

In this section we start off by recalling the basic theory of generic
morphisms and functors which admit them. We then define generic bicategories
and consider the properties of generic morphisms in these generic
bicategories. After discussing the coherence properties of these generic
morphisms, we go on to give the main result of this paper; showing
that the functors which respect these generic morphisms are precisely
the oplax functors. 

\subsection{Generic morphisms and factorisations}

Generic morphisms (and weaker analogues of them) have historically
arisen in the characterization the analytic endofunctors of $\mathbf{Set}$
\cite{Joyal86}, as well as the study of qualitative domains \cite{Gir86,Lam88}.
Characterizations of endofunctors which admit them have been studied
by Weber \cite{WeberGeneric}, and this is known to be related to
familial representability as studied by Diers \cite{Diers}.

In this paper we do not consider arbitrary endofunctors which admit
generics, but instead composition functors which admit generics, giving
us a richer structure to consider.
\begin{defn}
Given a functor $T\colon\mathcal{A}\to\mathcal{B}$ between categories
$\mathcal{A}$ and $\mathcal{B}$, we say a morphism $\delta\colon B\to TA$
in $\mathcal{B}$ (where $A\in\mathcal{A}$ and $B\in\mathcal{B}$)
is\emph{ $T$-generic} if for any commutative square of the form below
\[
\xymatrix{B\myard{\delta}{d}\myar{f}{r} & TC\ar[d]^{Tg}\\
TA\ar[r]_{Th}\ar@{..>}[ur]|-{T\overline{f}} & TD
}
\]
there exists a unique morphism $\overline{f}$ in $\mathcal{A}$ such
that $T\overline{f}\cdot\delta=f$.
\end{defn}

\begin{rem}
These are precisely the diagonally universal morphisms of Diers \cite{DiersDiag},
who noted that it must follow $g\cdot\overline{f}=h$ since both fillers
below
\[
\xymatrix{B\myard{\delta}{d}\myar{Tg\cdot f}{r} & TD\ar[d]^{T1_{D}} &  & B\myard{\delta}{d}\myar{Tg\cdot f}{r} & TD\ar[d]^{T1_{D}}\\
TA\ar[r]_{Th}\ar@{..>}[ur]|-{T\left(g\cdot\overline{f}\right)} & TD &  & TA\ar[r]_{Th}\ar@{..>}[ur]|-{Th} & TD
}
\]
render commutative the top triangles.
\end{rem}

\begin{defn}
We say a functor $T\colon\mathcal{A}\to\mathcal{B}$ between categories
$\mathcal{A}$ and $\mathcal{B}$ \emph{admits generic factorisations}
if for any morphism $f\colon B\to TC$ in $\mathcal{B}$ there exists
a $T$-generic morphism $\delta\colon B\to TA$ in $\mathcal{B}$
and morphism $\overline{f}\colon A\to C$ in $\mathcal{A}$ rendering
commutative
\[
\xymatrix@=1em{ & TA\ar[rd]^{T\overline{f}}\\
B\ar[ur]^{\delta}\ar[rr]_{f} &  & TC
}
\]
\end{defn}

We are now ready to define generic bicategories, the structures to
be considered in this paper. It will be helpful to write composition
in diagrammatic order, denoted by the symbol ``;''.
\begin{defn}
We say a bicategory $\mathscr{A}$ is \emph{generic }if for every
triple of objects $X,Y,Z\in\mathscr{A}$ the composition functor
\[
\xymatrix{\mathscr{A}_{X,Y}\times\mathscr{A}_{Y,Z}\myar{;}{r} & \mathscr{A}_{X,Z}}
\]
admits generic factorisations. Moreover, we simply call \emph{generic}
those 2-cells $\delta\colon c\to l;r$ which are ;-generic. 
\end{defn}

\begin{rem}
Unpacking the above definition into a more useful form, we see that
a 2-cell $\delta\colon c\to l;r$ is generic if and only if every
commuting diagram of the form
\[
\xymatrix{c\ar[d]_{\delta}\myar{\gamma}{r} & f;g\ar[d]^{\phi_{1};\phi_{2}}\\
l;r\myard{\theta_{1};\theta_{2}}{r}\ar@{..>}[ur]|-{\gamma_{1};\gamma_{2}} & m;n
}
\]
(where $\theta_{1},\theta_{2},\phi_{1},\phi_{2}$ and $\gamma$ are
arbitrary 2-cells) admits a filler $\gamma_{1};\gamma_{2}$ as displayed,
such that the top triangle commutes and the bottom triangle commutes
component-wise. Moreover, the pair $\left(\gamma_{1},\gamma_{2}\right)$
must be unique such that the top triangle commutes, justifying the
notation.
\end{rem}

\begin{rem}
As we will see in Section \ref{consequences}, there are a number
of well known bicategories and monoidal categories which are generic,
such as:
\begin{itemize}
\item any cartesian monoidal category;
\item finite sets and bijections with the disjoint union monoidal structure;
\item the bicategory of spans;
\item the bicategory of polynomials with cartesian 2-cells.
\end{itemize}
\end{rem}

Generic bicategories may be alternatively defined in terms of familial
representability, a property which is often easier to verify. This
is a consequence of the following known relationship\footnote{We include the proof of this relationship due to the difficulty of finding a reference.}
between functors which admit generics and the familial representability
conditions of Diers \cite{Diers}. 
\begin{prop}
[Diers] Given a functor $T\colon\mathcal{A}\to\mathcal{B}$ between
categories $\mathcal{A}$ and $\mathcal{B}$ the following are equivalent:
\begin{enumerate}
\item the functor $T$ admits generic factorisations;
\item for every $B\in\mathcal{B}$ there exists a set $\mathfrak{M}_{B}$
and function $P_{\left(-\right)}\colon\mathfrak{M}_{B}\to\mathcal{A}_{\textnormal{ob}}$
yielding isomorphisms 
\[
\mathcal{B}\left(B,TA\right)\cong\sum_{\delta\in\mathfrak{M}_{B}}\mathcal{A}\left(P_{\delta},A\right)
\]
natural in $A\in\mathcal{A}$.
\end{enumerate}
\end{prop}

\begin{proof}
Suppose that $T$ admits generic factorisations. Call two generic
morphisms $\delta$ and $\delta'$ equivalent if there exists an isomorphism
$\alpha$ rendering commutative a diagram as below:
\[
\xymatrix@=1em{TM\ar[rr]^{T\alpha} &  & TM'\\
 & B\ar[ru]_{\delta}\ar[ul]^{\delta'}
}
\]
Now take $\mathfrak{M}_{B}$ to be the set of equivalence classes
of generic morphisms out of $B$, with each class labeled by a chosen
representative. It follows that for any $f\colon B\to TA$ we can
find a representative generic morphism $\delta_{f}$ and unique morphism
$\overline{f}$ rendering commutative
\[
\xymatrix@=1em{B\ar[rr]^{f}\ar[rd]_{\delta_{f}} &  & TA\\
 & TM\ar[ru]_{T\overline{f}}
}
\]
We note also that the representative generic $\delta_{f}$ is itself
unique (such a generic necessarily lies in the same equivalence class).
Therefore the assignment $f\mapsto\left(\delta_{f},\overline{f}\right)$
is bijective, where each $P_{\delta_{f}}$ is taken as the $M$ above.
Trivially, given a map $x\colon A\to A'$ the diagram 
\[
\xymatrix@=1.5em{B\ar[rr]^{f}\ar[rdr]_{\delta_{f}} &  & TA\ar[rr]^{Tx} &  & TA'\\
 &  & TM\ar[u]_{T\overline{f}}\ar[urr]_{T\left(x\overline{f}\right)}
}
\]
commutes, and by genericity $x\overline{f}$ is the unique such map
making the outside commute; thus showing naturality.

Conversely, suppose we are given such a family of isomorphisms\footnote{Here $\mathfrak{M}_B$ is an arbitrary set, so we do not use the suggestive notation $\delta$ for its elements.}
\[
\mathcal{B}\left(B,TA\right)\cong\sum_{m\in\mathfrak{M}_{B}}\mathcal{A}\left(P_{m},A\right)
\]
natural in $A\in\mathcal{A}$, where $B\in\mathcal{B}$ is given.
We first note that by naturality, the inverse assignment is necessarily
defined by
\[
\xymatrix@=0.5em{m\in\mathfrak{M}_{B} & , & P_{m}\myar{\alpha}{rr} &  & A & \mapsto & B\myar{\delta_{m}}{rr} &  & TP_{m}\myar{T\alpha}{rr} &  & TA}
\]
where $\delta_{m}$ is the morphism corresponding to the identity
at $P_{m}$. Also, this $\delta_{m}$ is generic since given any commuting
diagram as on the outside below
\[
\xymatrix{B\myar{f}{r}\ar[d]_{\delta_{m}} & TA\ar[d]^{Th}\\
TP_{m}\myard{Tg}{r}\ar@{..>}[ru]|-{T\overline{f}} & TD
}
\]
the morphism $Th\cdot f$ must correspond to the pair $\left(\delta_{m},g\right)$
under the bijection. By naturality, $f$ must factor through this
same $\delta_{m}$, and so the pair $\left(\delta_{m},\overline{f}\right)$
corresponding to $f$ is unique such that the top triangle commutes.
That $g=h\cdot\overline{f}$ is also a consequence of naturality.
It is implicit in the above argument that $T$ then admits generic
factorisations.
\end{proof}
Taking $T$ to be the composition functor, we have the following.
\begin{cor}
A bicategory $\mathscr{A}$ is generic if and only if for any triple
of objects $X,Y,Z\in\mathscr{A}$ and 1-cell $c\colon X\to Z$ the
functor
\[
\mathscr{A}_{X,Z}\left(c,-;-\right)\colon\mathscr{A}_{X,Y}\times\mathscr{A}_{Y,Z}\to\mathbf{Set}
\]
is a coproduct of representables, meaning that for any $\left(X,Y,Z,c\right)$
there exists a set $\mathfrak{M}_{c}^{X,Y,Z}$ equipped with projections
\[
\xymatrix{\left(\mathscr{A}_{X,Y}\right)_{\textnormal{ob}} & \mathfrak{M}_{c}^{X,Y,Z}\myar{r_{\left(-\right)}}{r}\myard{l_{\left(-\right)}}{l} & \left(\mathscr{A}_{Y,Z}\right)_{\textnormal{ob}}}
\]
such that for all $a\colon X\to Y$ and $b\colon Y\to Z$ we have
isomorphisms
\begin{equation}
\mathscr{A}_{X,Z}\left(c,a;b\right)\cong\sum_{m\in\mathfrak{M}_{c}^{X,Y,Z}}\mathscr{A}_{X,Y}\left(l_{m},a\right)\times\mathscr{A}_{Y,Z}\left(r_{m},b\right)\label{coprodrep}
\end{equation}
natural in $a$ and $b$. 
\end{cor}

We have defined generics as universal maps into a composite of two
1-cells; what one might call ``2-generics''. We might ask if there
is a corresponding notion for ``0-generics'' into composites of
zero 1-cells, that is, identity 1-cells. However, as for each $n\colon X\to X$
the functor
\[
\mathscr{A}_{X,X}\left(n,1_{X}\right)\colon\mathbf{1}\to\mathbf{Set}
\]
is trivially a coproduct of representables, there is no condition
to impose on these 2-cells, and so any 2-cell $\epsilon\colon n\to1_{X}$
may be regarded as a ``0-generic''. Regardless, these 2-cells still
have an interesting property; they are unique.
\begin{prop}
Suppose $\mathscr{A}$ is a generic bicategory. Then for each $X\in\mathscr{A}$,
the identity 1-cell $1_{X}$ is sub-terminal in $\mathscr{A}_{X,X}$.
\end{prop}

\begin{proof}
Given a morphism $n\colon X\to X$ and two 2-cells $s,t\colon n\to1_{X}$
we have two commuting squares
\[
\xymatrix{n\myar{\delta_{1}}{r}\ar[d]_{\delta_{2}} & l;n\ar[d]^{h;s} &  & n\myar{\delta_{1}}{r}\ar[d]_{\delta_{2}} & l;n\ar[d]^{h;t}\\
n;r\myard{s;k}{r}\ar@{..>}[ru]|-{\theta;\phi} & 1_{X};1_{X} &  & n;r\myard{t;k}{r}\ar@{..>}[ru]|-{\theta;\phi} & 1_{X};1_{X}
}
\]
where $\delta_{1}$ and $h\colon l\to1_{X}$ are given by factorizing
the unitor $n\to1_{X};n$ through a generic, and $\delta_{2}$ and
$k\colon r\to1_{X}$ are given by factorizing the other unitor $n\to n;1_{X}$.
Now both of these squares admit a unique filler, and moreover both
these fillers must be equal as uniqueness is forced by the top left
triangles; we denote this filler $\theta;\phi$. Equating the left
components of the bottom right triangles we then find $s=h\theta=t$.
\end{proof}
It will be useful to give such 2-cells a name as they still play an
important role, despite the lack of a non-trivial universal property.
\begin{defn}
We call any 2-cell of the form $\epsilon\colon n\to1_{X}$ in a bicategory
$\mathscr{A}$ an \emph{augmentation}.
\end{defn}

\subsection{Coherence of generics}

The following two lemmata show that there exists ``nice'' choices
of generics. This will later be useful in regard to stating and checking
coherence conditions.
\begin{lem}
\label{cohunit} Suppose $\mathscr{A}$ is a generic bicategory. Then
for any factorization of a left unitor at a 1-cell $c\colon X\to Y$
through a generic $\delta$ as below
\begin{equation}
\xymatrix@=1em{ & l;r\ar[dr]^{\theta;\phi}\\
c\myard{\textnormal{unitor}}{rr}\ar[ur]^{\delta} &  & 1_{X};c
}
\label{unitlemma}
\end{equation}
the induced 2-cell $\phi$ is invertible. 
\end{lem}

\begin{proof}
Define $\phi^{\ast}\colon c\to r$ to be the composite
\[
\xymatrix{c\myar{\delta}{r} & l;r\myar{\theta;r}{r} & 1_{X};r\myar{\textnormal{unitor}}{r} & r}
\]
and note that when this is post-composed by $\phi$ we recover the
identity 2-cell at $c$, by commutativity of the diagram \ref{unitlemma}
and naturality of unitors. We also note that by naturality of unitors
the diagram
\[
\xymatrix{c\myar{\textnormal{unitor}}{r}\ar[d]_{\delta} & 1_{X};c\ar[d]^{1_{X};\phi^{\ast}}\\
l;r\myard{\theta;r}{r}\ar@{..>}[ur]|-{\theta;\phi} & 1_{X};r
}
\]
 commutes and thus admits a filler such that both triangles commute.
Moreover, we note that as uniqueness is forced by the top triangle
this filler must be $\theta;\phi$. Equating the second components
of the bottom right triangle we have established $\phi$ followed
by $\phi^{\ast}$ as being the identity. 
\end{proof}
\begin{rem}
As $\phi$ is invertible above, composing the generic $\delta$ with
$\phi$ still yields a generic. This shows that there exists ``nice''
generics $c\to l;c$ and augmentations $l\to1_{X}$ which compose
to the unitor. Moreover, it is clear this may be similarly done for
right unitors.
\end{rem}

\begin{lem}
\label{cohassoc} Suppose $\mathscr{A}$ is a generic bicategory.
Let $W,X,Y,Z$ be objects in $\mathscr{A}$, let $T$ be the functor
given by composition
\[
\left(\mathscr{A}_{W,X}\times\mathscr{A}_{X,Y}\right)\times\mathscr{A}_{Y,Z}\to\mathscr{A}_{W,Y}\times\mathscr{A}_{Y,Z}\to\mathscr{A}_{W,Z}
\]
and consider 1-cells
\[
d\colon W\to Z,\qquad l\colon W\to X,\qquad m\colon X\to Y,\qquad r\colon Y\to Z.
\]
Then a 2-cell $d\to\left(l;m\right);r$ in $\mathscr{A}$ is $T$-generic
if and only if it has the form
\[
\xymatrix{d\myar{\delta_{1}}{r} & h;r\myar{\delta_{2};r}{r} & \left(l;m\right);r}
\]
for a pair of generics $\delta_{1}$ and $\delta_{2}$.
\end{lem}

\begin{proof}
Suppose we are given generics $\delta_{1}$ and $\delta_{2}$ composable
as in the diagram on the left below
\[
\xymatrix@=1em{d\ar[d]_{\delta_{1}}\myar{\gamma}{rr} &  & \left(a;b\right);c\ar[dd]^{\left(\beta_{1};\beta_{2}\right);\beta_{3}} &  &  &  & h\myar{\gamma_{1}}{rr}\ar[dd]_{\delta_{2}} &  & a;b\ar[dd]^{\beta_{1};\beta_{2}}\\
h;r\ar[d]_{\delta_{2};r}\ar@{..>}[urr]_{\gamma_{1};\gamma_{2}}\\
\left(l;m\right);r\myard{\left(\alpha_{1};\alpha_{2}\right);\alpha_{3}}{rr} &  & \left(f;g\right);h &  &  &  & l;m\myard{\alpha_{1};\alpha_{2}}{rr}\ar@{..>}[uurr]|-{\zeta_{1};\zeta_{2}} &  & f;g
}
\]
where $\alpha_{1},\alpha_{2},\alpha_{3},\beta_{1},\beta_{2},\beta_{3}$
and $\gamma$ are arbitrary 2-cells such that the outside diagram
commutes. Then there exists a filler $\gamma_{1};\gamma_{2}$ splitting
the diagram into two commuting regions, by genericity of $\delta_{1}$.
Moreover, there exists a filler $\zeta_{1};\zeta_{2}$ for the commuting
diagram on the right above as $\delta_{2}$ is generic. We thus have
a diagonal filler $\left(\zeta_{1};\zeta_{2}\right);\gamma_{2}$ for
the diagram on the left above. For uniqueness, suppose we are given
another filler $\left(\zeta_{1}';\zeta_{2}'\right);\gamma_{2}'$ and
note that since $\delta_{1}$ is generic, we have $\left[\left(\zeta_{1}';\zeta_{2}'\right)\circ\delta_{2}\right];\gamma_{2}'=\gamma_{1};\gamma_{2}$
component wise. Hence $\gamma_{2}'=\gamma_{2}$ and $\left(\zeta_{1}';\zeta_{2}'\right)\circ\delta_{2}=\gamma_{1}$.
Since $\delta_{2}$ is generic it follows that $\zeta_{1}'=\zeta_{1}$
and $\zeta_{2}'=\zeta_{2}$.

Conversely, suppose we are given a 2-cell $\delta\colon d\to\left(l;m\right);r$
which is $T$-generic. Now, we know that the $T$-generic $\delta$
can be factored through a generic $\delta_{1}$ giving the triangle
on the left below
\[
\xymatrix{d\myar{\delta_{1}}{r}\ar[d]_{\delta} & h';r'\myar{\delta_{2};r'}{r}\ar@{..>}[dl]|-{\alpha;\beta} & \left(l';m'\right);r'\ar@{..>}[dll]^{\left(\gamma_{1};\gamma_{2}\right);\beta}\\
\left(l;m\right);r
}
\]
and the 2-cell $\alpha$ can be factored through a generic $\delta_{2}$
yielding the right triangle above. In particular, the components of
$\left(\gamma_{1};\gamma_{2}\right);\beta$ are invertible as this
is an induced isomorphism of $T$-generic morphisms \cite[Lemma 5.7]{WeberGeneric}.
Hence upon taking $\delta_{1}^{\ast}$ to be $\delta_{1}$ pasted
with $\beta$, and $\delta_{2}^{\ast}$ to be $\delta_{2}$ pasted
with $\gamma_{1};\gamma_{2}$, we see that $\delta$ is a pasting
of generics $\delta_{1}^{\ast}$ and $\delta_{2}^{\ast}$.
\end{proof}
\begin{rem}
The above lemma is an instance of a more general fact: if $\delta_{1}\colon C\to SB$
is $S$-generic and $\delta_{2}\colon B\to TA$ is $T$-generic, then
\[
\xymatrix{C\ar[r]^{\delta_{1}} & SB\ar[r]^{S\delta_{2}} & STA}
\]
 is $ST$-generic. Moreover, if both $S$ and $T$ admit generic factorisations
then all $ST$-generics have this form.
\end{rem}

\begin{rem}
Clearly, we can state and prove an analogue of the above lemma if
we replace $T$ by the functor $S$ given as the composite
\[
\mathscr{A}_{W,X}\times\left(\mathscr{A}_{X,Y}\times\mathscr{A}_{Y,Z}\right)\to\mathscr{A}_{W,Y}\times\mathscr{A}_{Y,Z}\to\mathscr{A}_{W,Z}
\]
It is also clear that given a composite of generics
\[
\xymatrix{d\myar{\delta_{1}}{r} & h;r\myar{\delta_{2};r}{r} & \left(l;m\right);r}
\]
which is $T$-generic, that the composite
\[
\xymatrix{d\myar{\delta_{1}}{r} & h;r\myar{\delta_{2};r}{r} & \left(l;m\right);r\myar{\textnormal{assoc}}{r} & l;\left(m;r\right)}
\]
is $S$-generic, and hence by the analogue of the above lemma we may
write this composite as
\[
\xymatrix{d\myar{\delta_{3}}{r} & l;k\myar{l;\delta_{4}}{r} & l;\left(m;r\right)}
\]
for some pair of generics $\delta_{3}$ and $\delta_{4}$.
\end{rem}

It is sometimes advantageous to not consider all generics, but only
a smaller class of generics satisfying some coherence properties outlined
in the following definition.
\begin{defn}
Let $\mathscr{A}$ be a generic bicategory. Let $\Delta_{2}$ and
$\Delta_{0}$ be given collections of generics and augmentations in
$\mathscr{A}$ respectively. Denote by $\Omega_{2}$ the set of domains
of the generics in $\Delta_{2}$. We say the pair $\left(\Delta_{2},\Delta_{0}\right)$
is \emph{coherent }if:
\begin{enumerate}
\item (completeness of generics) for every generic $\delta'\colon c'\to l';r'$
in $\mathscr{A}$ there exists a generic $\delta\colon c\to l;r$
in $\Delta_{2}$ and isomorphisms $\zeta_{1},\zeta_{2}$ and $\zeta$
rendering commutative
\[
\xymatrix{c\myar{\delta}{r}\ar[d]_{\zeta} & l;r\ar[d]^{\zeta_{1};\zeta_{2}}\\
c'\myard{\delta'}{r} & l';r'
}
\]
\item (completeness of augmentations) for every augmentation $\epsilon'\colon n'\to1_{X}$
in $\mathscr{A}$ there exists an augmentation $\epsilon\colon n\to1_{X}$
in $\Delta_{0}$ and isomorphism $\xi\colon n\to n'$ rendering commutative
\[
\xymatrix@=1em{n\myar{\xi}{rr}\ar[rd]_{\epsilon} &  & n'\ar[ld]^{\epsilon'}\\
 & 1_{X}
}
\]
\item (associator coherence) for all generics $\delta_{1},\delta_{2}\in\Delta_{2}$
composable as below, there exists generics $\delta_{3},\delta_{4}\in\Delta_{2}$
rendering commutative
\[
\xymatrix@=1em{c\ar[d]_{\delta_{3}}\ar@{=}[rr] &  & c\ar[d]^{\delta_{1}}\\
l;k\myard{l;\delta_{4}}{d} &  & h;r\ar[d]^{\delta_{2};r}\\
l;\left(m;r\right)\myard{\textnormal{assoc}}{rr} &  & \left(l;m\right);r
}
\]
\item (left unitor coherence) for all $c\colon X\to Y$ in $\Omega_{2}$
there exists a $\delta\in\Delta_{2}$ and $\epsilon\in\Delta_{0}$
composable as below and rendering commutative
\[
\xymatrix@=1em{ & n;c\ar[dr]^{\epsilon;c}\\
c\myard{\textnormal{unitor}}{rr}\ar[ur]^{\delta} &  & 1_{X};c
}
\]
\item (right unitor coherence) for all $c\colon X\to Y$ in $\Omega_{2}$
there exists a $\delta\in\Delta_{2}$ and $\epsilon\in\Delta_{0}$
composable as below and rendering commutative
\[
\xymatrix@=1em{ & c;n\ar[dr]^{c;\epsilon}\\
c\myard{\textnormal{unitor}}{rr}\ar[ur]^{\delta} &  & c;1_{Y}
}
\]
\end{enumerate}
\end{defn}

\begin{rem}
If $\mathscr{A}$ is generic, we may always take $\left(\Delta_{2},\Delta_{0}\right)$
to be the class of all generic 2-cells and augmentations. This is
a consequence of the previous two lemmata.
\end{rem}

\begin{rem}
Informally, the conditions (3) to (5) guarantee that each 1-cell $c\in\Omega_{2}$
admits the structure of an ``$\mathscr{A}$-comonoid''; a simple
example of this being that objects in cartesian monoidal categories
admit the structure of a comonoid.
\end{rem}

\subsection{Functors which respect generics}

It is well known that to give an oplax functor $L\colon\mathbf{1}\to\mathscr{C}$
is to give a comonad in $\mathscr{C}$. The following theorem generalizes
this fact, replacing the terminal category by any generic bicategory
$\mathscr{A}$.

At the same time, the following theorem may be seen as a coherence
result; it provides a reduction in the data of an oplax functor out
of such an $\mathscr{A}$, showing that the coherence data of such
an oplax functor is completely determined by the data at the diagonals.

The most important property of this result however is that it provides
a description of oplax functors $L\colon\mathscr{A}\to\mathscr{C}$
out of generic bicategories $\mathscr{A}$ which does not involve
composition in the domain bicategory; by this we mean expressions
of the form $L\left(a;b\right)$ or $L\left(1_{X}\right)$ do not
appear in our description below.

For completeness, we also give a reduced description of oplax natural
transformations and icons \cite{icons} between such oplax functors.
\begin{thm}
\label{generalcoherence} Let $\mathscr{A}$ and $\mathscr{C}$ be
bicategories, and suppose $\mathscr{A}$ is generic. Suppose we are
given a coherent class $\left(\Delta_{2},\Delta_{0}\right)$ of generics
and augmentations of $\mathscr{A}$. Then given a locally defined
functor
\[
L_{X,Y}\colon\mathscr{A}_{X,Y}\to\mathscr{C}_{LX,LY},\qquad X,Y\in\mathscr{A}
\]
the following data are in bijection:
\begin{enumerate}
\item for every pair of composable 1-cells $a$ and $b$, a constraint 2-cell
\[
\varphi_{a,b}\colon L\left(a;b\right)\to L\left(a\right);L\left(b\right)
\]
and for every identity 1-cell $1_{X}$, a constraint 2-cell 
\[
\lambda_{X}\colon L\left(1_{X}\right)\to1_{LX}
\]
exhibiting $L$ as an oplax functor;
\item for every generic $\delta\colon c\to l;r$ in $\Delta_{2}$, a comultiplication
2-cell
\[
\Phi_{\delta}\colon L\left(c\right)\to L\left(l\right);L\left(r\right)
\]
and for every augmentation $\epsilon\colon n\to1_{X}$ in $\Delta_{0}$,
a counit 2-cell
\[
\Lambda_{\epsilon}\colon L\left(n\right)\to1_{LX}
\]
satisfying the following coherence axioms:
\begin{enumerate}
\item (naturality of comultiplication) for any 2-cell $\zeta\colon c\to c'$
and commuting diagram as on the left below\footnote{The 2-cells $\zeta_1$ and $\zeta_2$ are then induced by the genericity of $\delta_1$.}
with $\delta_{1},\delta_{2}\in\Delta_{2}$
\[
\xymatrix{ & c\myar{\delta_{1}}{r}\ar[d]_{\zeta} & l;r\ar[d]^{\zeta_{1};\zeta_{2}} &  & Lc\myar{\Phi_{\delta_{1}}}{r}\ar[d]_{L\zeta} & Ll;Lr\ar[d]^{L\zeta_{1};L\zeta_{2}}\\
 & c'\myard{\delta_{2}}{r} & l';r' &  & Lc'\myard{\Phi_{\delta_{2}}}{r} & Ll';Lr'
}
\]
the diagram on the right above commutes;
\item (naturality of counits) for any 2-cell $\xi\colon n\to n'$ and pair
of augmentations $\epsilon\colon n\to1_{X}$ and $\epsilon'\colon n'\to1_{X}$
in $\Delta_{0}$ giving a commuting diagram as on the left below
\[
\xymatrix@=1em{ &  & n\myar{\xi}{rr}\ar[rd]_{\epsilon} &  & n'\ar[ld]^{\epsilon'} &  &  & Ln\myar{L\xi}{rr}\ar[rd]_{\Lambda_{\epsilon}} &  & Ln'\ar[ld]^{\Lambda_{\epsilon'}}\\
 &  &  & 1_{X} &  &  &  &  & 1_{LX}
}
\]
the diagram on the right above commutes;
\item (associativity of comultiplication) for every $\delta_{1},\delta_{2},\delta_{3},\delta_{4}\in\Delta_{2}$
yielding an equality as on the left below
\[
\xymatrix@=1em{ &  & c\ar[d]_{\delta_{3}}\ar@{=}[rr] &  & c\ar[d]^{\delta_{1}} &  & Lc\ar[d]_{\Phi_{\delta_{3}}}\ar@{=}[rr] &  & Lc\ar[d]^{\Phi_{\delta_{1}}}\\
 &  & l;k\myard{l;\delta_{4}}{d} &  & h;r\ar[d]^{\delta_{2};r} &  & Ll;Lk\myard{Ll;\Phi_{\delta_{4}}}{d} &  & Lh;Lr\ar[d]^{\Phi_{\delta_{2}};Lr}\\
 &  & l;\left(m;r\right)\myard{\textnormal{assoc}}{rr} &  & \left(l;m\right);r &  & Ll;\left(Lm;Lr\right)\myard{\textnormal{assoc}}{rr} &  & \left(Ll;Lm\right);Lr
}
\]
the diagram on the right above commutes;
\item (left counit axiom) for any 1-cell $c\colon X\to Y$, generic $\delta\in\Delta_{2}$
and augmentation $\epsilon\in\Delta_{0}$ yielding an equality as
on the left below 
\[
\xymatrix@=1em{ &  &  & n;c\ar[dr]^{\epsilon;c} &  &  &  &  & Ln;Lc\ar[dr]^{\Lambda_{\epsilon};Lc}\\
 &  & c\myard{\textnormal{unitor}}{rr}\ar[ur]^{\delta} &  & 1_{X};c &  &  & Lc\myard{\textnormal{unitor}}{rr}\ar[ur]^{\Phi_{\delta}} &  & 1_{LX};Lc
}
\]
the diagram on the right above commutes;
\item (right counit axiom) for any 1-cell $c\colon X\to Y$, generic $\delta\in\Delta_{2}$
and augmentation $\epsilon\in\Delta_{0}$ yielding an equality as
on the left below 
\[
\xymatrix@=1em{ &  &  & c;n\ar[dr]^{c;\epsilon} &  &  &  &  & Lc;Ln\ar[dr]^{Lc;\Lambda_{\epsilon}}\\
 &  & c\myard{\textnormal{unitor}}{rr}\ar[ur]^{\delta} &  & c;1_{Y} &  &  & Lc\myard{\textnormal{unitor}}{rr}\ar[ur]^{\Phi_{\delta}} &  & Lc;1_{LY}
}
\]
the diagram on the right above commutes.
\end{enumerate}
\end{enumerate}
Suppose now we are given a locally defined functor $L$ equipped with
a collection $\left(\varphi,\lambda\right)$ as in (1), or equivalently
equipped with a collection $\left(\Phi,\Lambda\right)$ as in (2).
Denote this data by the 5-tuple $\left(L,\varphi,\Phi,\lambda,\Lambda\right)$
whilst noting the collections $\left(\varphi,\lambda\right)$ and
$\left(\Phi,\Lambda\right)$ uniquely determine each other. Let $\left(K,\psi,\Psi,\gamma,\Gamma\right)$
be another such 5-tuple. Then the following data are in bijection:
\begin{enumerate}
\item an oplax natural transformation $\vartheta\colon L\implies K$ of
oplax functors;
\item for every object $X\in\mathscr{A}$, a 1-cell $\vartheta_{X}\colon LX\to KX$
in $\mathscr{C}$, and for every 1-cell $f\colon X\to Y$ in $\mathscr{A}$,
a 2-cell
\[
\xymatrix{LX\myar{Lf}{r}\ar[d]_{\vartheta_{X}}\ar@{}[dr]|-{\Downarrow\vartheta_{f}} & LY\ar[d]^{\vartheta_{Y}}\\
KX\myard{Kf}{r} & KY
}
\]
natural in 1-cells $f\colon X\to Y$ and satisfying the following
conditions:
\begin{enumerate}
\item for every generic $\delta\colon c\to l;r$ in $\Delta_{2}$,
\[
\xymatrix{LX\myar{Lc}{rr}\ar[d]_{\vartheta_{X}}\ar@{}[drr]|-{\Downarrow\vartheta_{c}} &  & LZ\ar[d]^{\vartheta_{Z}} &  & LX\myar{Lc}{rr}\ar[d]_{\vartheta_{X}}\ar@/_{1pc}/[rd]^{Ll} & \ar@{}[d]|-{\Downarrow\Phi_{\delta}} & LZ\ar[d]^{\vartheta_{Z}}\\
KX\myar{Kc}{rr}\ar@/_{1pc}/[rd]_{Kl} & \ar@{}[d]|-{\Downarrow\Psi_{\delta}} & KZ & = & KX\ar@/_{1pc}/[rd]_{Kl}\ar@{}[dr]|-{\Downarrow\vartheta_{l}} & LY\ar@/_{1pc}/[ru]^{Lr}\ar[d]|-{\theta_{Y}} & KZ\ar@{}[dl]|-{\Downarrow\vartheta_{r}}\\
 & KY\ar@/_{1pc}/[ur]_{Kr} &  &  &  & KY\ar@/_{1pc}/[ur]_{Kr}
}
\]
\item for every augmentation $\epsilon\colon n\to1_{X}$ in $\Delta_{0}$,
\[
\xymatrix{LX\myar{Ln}{rr}\ar[d]_{\vartheta_{X}}\ar@{}[drr]|-{\Downarrow\vartheta_{n}} &  & LX\ar[d]^{\vartheta_{X}} &  & LX\myar{Ln}{rr}\ar[d]_{\vartheta_{X}}\ar@/_{2.5pc}/[rr]_{1_{LX}} & \ar@{}[d]|-{\Downarrow\Lambda_{\epsilon}} & LX\ar[d]^{\vartheta_{X}}\\
KX\myar{Kn}{rr}\ar@/_{2.5pc}/[rr]_{1_{KX}} & \ar@{}[d]|-{\Downarrow\Gamma_{\epsilon}} & KX & = & KX\ar@/_{2.5pc}/[rr]_{1_{KX}} & \;\ar@{}[d]|-{\Downarrow\textnormal{id}} & KX\\
 & \; &  &  &  & \;
}
\]
\end{enumerate}
\end{enumerate}
When $L$ and $K$ agree on objects, this restricts to the bijection
of the following data:
\begin{enumerate}
\item An icon between oplax functors 
\[
\vartheta\colon L\implies K\colon\mathscr{A}\to\mathscr{C}
\]
\item A collection of natural transformations 
\[
\vartheta_{X,Y}\colon L_{X,Y}\implies K_{X,Y}\colon\mathscr{A}_{X,Y}\to\mathscr{C}_{X,Y},\qquad X,Y\in\mathscr{A}
\]
rendering commutative the diagrams
\[
\xymatrix{L\left(c\right)\myar{\Phi_{\delta}}{rr}\ar[d]_{\vartheta_{c}} &  & L\left(l\right);L\left(r\right)\ar[d]^{\vartheta_{l};\vartheta_{r}} & L\left(n\right)\myar{\vartheta_{n}}{rr}\ar[rd]_{\Lambda_{n}} &  & K\left(n\right)\ar[ld]^{\Gamma_{n}}\\
K\left(c\right)\myard{\Psi_{\delta}}{rr} &  & K\left(l\right);K\left(r\right) &  & 1_{X}
}
\]
\end{enumerate}
\end{thm}

\begin{proof}
We divide the proof into parts, verifying each bijection separately.

\noun{Bijection With Oplax Functors.} We first show how to pass between
the data of (1) and (2), and then verify this defines a bijection.

$\left(1\right)\Longrightarrow\left(2\right)\colon$ Suppose we are
given the data $\left(L,\varphi,\lambda\right)$ of (1). We define
$\Phi_{\delta}$ for each generic $\delta\colon c\to l;r$ by the
composite
\begin{equation}
\xymatrix{L\left(c\right)\myar{L\delta}{r} & L\left(l;r\right)\myar{\varphi_{l,r}}{r} & L\left(l\right);L\left(r\right)}
\label{defPhi}
\end{equation}
and define $\Lambda_{\epsilon}$ for each augmentation $\epsilon\colon n\to1_{X}$
by the composite
\begin{equation}
\xymatrix{L\left(n\right)\myar{L\epsilon}{r} & L\left(1_{X}\right)\myar{\lambda_{X}}{r} & 1_{LX}}
\label{defLambda}
\end{equation}
For naturality of comultiplication, we see that given a diagram as
on the left below
\[
\xymatrix{c\myar{\delta_{1}}{r}\ar[d]_{\zeta} & l;r\ar[d]^{\zeta_{1};\zeta_{2}} &  & Lc\myar{L\delta_{1}}{r}\ar[d]_{L\zeta} & L\left(l;r\right)\myar{\varphi_{l,r}}{r}\ar[d]|-{L\left(\zeta_{1};\zeta_{2}\right)} & Ll;Lr\ar[d]^{L\zeta_{1};L\zeta_{2}}\\
c'\myard{\delta_{2}}{r} & l';r' &  & Lc'\myard{L\delta_{2}}{r} & L\left(l';r'\right)\myard{\varphi_{l',r'}}{r} & Ll';Lr'
}
\]
the right commutes by naturality of $\varphi$ and local functoriality
of $L$. For naturality of counits note that given a commuting diagram
as on the left below
\[
\xymatrix@=1em{n\ar[dd]_{\xi}\ar[rd]^{\epsilon} &  &  &  &  & Ln\ar[dd]_{L\xi}\ar[rd]^{L\epsilon}\\
 & 1_{X} &  &  &  &  & L1_{X}\ar[r]^{\lambda_{X}} & 1_{LX}\\
n'\ar[ru]_{\epsilon'} &  &  &  &  & Ln'\ar[ru]_{L\epsilon'}
}
\]
the right trivially commutes. For associativity of comultiplication,
note that given a commuting diagram
\[
\xymatrix@=1em{c\ar[d]_{\delta_{3}}\ar@{=}[rr] &  & c\ar[d]^{\delta_{1}}\\
l;k\myard{l;\delta_{4}}{d} &  & h;r\ar[d]^{\delta_{2};r}\\
l;\left(m;r\right)\myard{\textnormal{assoc}}{rr} &  & \left(l;m\right);r
}
\]
we have the commutativity of the diagram
\[
\xymatrix@=1em{Lc\ar[d]_{L\delta_{3}}\ar@{=}[rrrr] &  &  &  & Lc\ar[d]^{L\delta_{1}}\\
L\left(l;k\right)\ar[d]_{\varphi_{l,k}}\ar[rd]^{L\left(l;\delta_{4}\right)} &  &  &  & L\left(h;r\right)\ar[d]^{\varphi_{h,r}}\ar[ld]_{L\left(\delta_{1};r\right)}\\
Ll;Lk\myard{Ll;L\delta_{4}}{d} & L\left(l;\left(m;r\right)\right)\ar[ld]^{\varphi_{l,\left(m;r\right)}}\myard{L\left(\textnormal{assoc}\right)}{rr} &  & L\left(\left(l;m\right);r\right)\ar[dr]_{\varphi_{\left(l;m\right),r}} & Lh;Lr\ar[d]^{L\delta_{2};Lr}\\
Ll;L\left(m;r\right)\ar[d]_{Ll;\varphi_{m,r}} &  &  &  & L\left(l;m\right);Lr\ar[d]^{\varphi_{l,m};Lr}\\
Ll;\left(Lm;Lr\right)\myard{\textnormal{assoc}}{rrrr} &  &  &  & \left(Ll;Lm\right);Lr
}
\]
by naturality of $\varphi$, associativity of $\varphi$ and local
functoriality of $L$. For the left counit axiom, suppose we are given
a commuting diagram as on the left below
\[
\xymatrix@=1.5em{ & l;c\ar[dr]^{\epsilon;c} &  &  & Lc\myar{\delta}{r}\ar@/_{1pc}/[rrd]_{L\left(\textnormal{unitor}\right)} & L\left(l;c\right)\myar{\varphi_{l,r}}{r}\ar[rd]_{L\left(\epsilon;c\right)} & Ll;Lc\myar{L\epsilon;Lc}{r} & L1_{X};Lc\myar{\lambda_{X};Lc}{r} & 1_{LX};Lc\\
c\myard{\textnormal{unitor}}{rr}\ar[ur]^{\delta} &  & 1_{X};c &  &  &  & L\left(1_{X};c\right)\ar[ur]_{\varphi_{_{1_{X}},c}}\myard{L\left(\textnormal{unitor}\right)}{r} & L\left(c\right)\ar@/_{0.7pc}/[ur]_{\textnormal{unitor}}
}
\]
and note the composite on the right above is the unitor by local functoriality
of $L$, naturality of $\varphi$, and the unitary axiom on $\lambda$.
The right counit axiom is similar.

$\left(2\right)\Longrightarrow\left(1\right)\colon$ Suppose we are
given the data $\left(L,\Phi,\Lambda\right)$ for a coherent class
$\left(\Delta_{2},\Delta_{0}\right)$. Now for any generic $\delta'\colon c'\to l';r'$
in $\mathscr{A}$ we have a commuting diagram as on the left below
with $\zeta_{1},\zeta_{2},\zeta$ invertible and $\delta\in\Delta_{2}$
\[
\xymatrix{c\myar{\delta}{r}\ar[d]_{\zeta} & l;r\ar[d]^{\zeta_{1};\zeta_{2}} &  & Lc\myar{\Phi_{\delta}}{r}\ar[d]_{L\zeta} & Ll;Lr\ar[d]^{L\zeta_{1};L\zeta_{2}}\\
c'\myard{\delta'}{r} & l';r' &  & Lc'\myard{\Phi_{\delta'}}{r} & Ll';Lr'
}
\]
and so we may define $\Phi_{\delta'}$ as the unique morphism making
the diagram on the right above commute; this being well defined as
a consequence of naturality of comultiplication.

Similarly, for any augmentation $\epsilon'\colon n'\to1_{X}$ in $\mathscr{A}$
there exists an augmentation $\epsilon\colon n\to1_{X}$ in $\Delta_{0}$
and isomorphism $\xi\colon n\to n'$ rendering commutative the left
diagram below
\[
\xymatrix@=1em{n\myar{\xi}{rr}\ar[rd]_{\epsilon} &  & n'\ar[ld]^{\epsilon'} &  &  & Ln\myar{L\xi}{rr}\ar[rd]_{\Lambda_{\epsilon}} &  & Ln'\ar[ld]^{\Lambda_{\epsilon'}}\\
 & 1_{X} &  &  &  &  & 1_{LX}
}
\]
and so we may define $\Lambda_{\epsilon'}$ as the unique morphism
making the right diagram above commute; similarly well defined by
naturality of counits.

We have now extended the definition of $\Phi$ and $\Lambda$ to all
generic morphisms and augmentations. Moreover, the naturality properties
now hold with respect to all generics $\delta$ and augmentations
$\epsilon$. Indeed, given any generics $\delta$ and $\delta'$ in
$\mathscr{A}$ and a diagram as on the left below (not assuming $\zeta,\zeta_{1}$
or $\zeta_{2}$ are invertible)
\[
\xymatrix@=1em{c\myar{\delta}{rr}\ar[ddd]_{\zeta} &  & l;r\ar[ddd]^{\zeta_{1};\zeta_{2}} &  &  &  &  & c\myar{\delta}{rr}\ar[d]_{\theta} &  & l;r\ar[d]^{\theta_{1};\theta_{2}}\\
 &  &  &  & \ar@{}[rd]|-{=} &  &  & \widetilde{c}\ar[d]_{\phi}\ar[rr]|-{\widetilde{\delta}} &  & \widetilde{l};\widetilde{r}\ar[d]^{\phi_{1};\phi_{2}}\\
 &  &  &  &  & \; &  & \widetilde{c'}\ar[d]_{\gamma}\ar[rr]|-{\widetilde{\delta'}} &  & \widetilde{l'};\widetilde{r'}\ar[d]^{\gamma_{1};\gamma_{2}}\\
c'\myard{\delta'}{rr} &  & l';r' &  &  &  &  & c'\myard{\delta'}{rr} &  & l';r'
}
\]
we can factor as on the right, where $\widetilde{\delta}$ and $\widetilde{\delta}'$
are in $\Delta_{2}$ and $\theta,\theta_{1},\theta_{2},\gamma,\gamma_{1}$
and $\gamma_{2}$ are invertible. Applying the naturality condition
to the three squares on the right then gives the naturality condition
for the left diagram. A similar calculation may be done concerning
augmentations.

To show that one may recover an oplax functor $L\colon\mathscr{A}\to\mathscr{C}$
we note we may define a general oplax constraint cell $\varphi_{a,b}\colon L\left(a;b\right)\to La;Lb$
by taking a diagram as on the left below with $\delta$ generic and
then defining the right diagram to commute.
\begin{equation}
\xymatrix@=1em{ & l;r\ar[rd]^{s_{1};s_{2}} &  &  &  &  & Ll;Lr\ar[rd]^{Ls_{1};Ls_{2}}\\
a;b\myard{\textnormal{id}}{rr}\ar[ur]^{\delta} &  & a;b &  &  & L\left(a;b\right)\myard{\varphi_{a,b}}{rr}\ar[ur]^{\Phi_{\delta}} &  & La;Lb
}
\label{defphi}
\end{equation}
Note that this is well defined since given two diagrams as on the
left above, we have a commuting diagram as on the left below
\begin{equation}
\xymatrix{a;b\ar[r]^{\delta}\ar[d]_{\delta'} & l;r\ar[d]^{s_{1};s_{2}} &  &  & La;Lb\ar[r]^{\Phi_{\delta}}\ar[d]_{\Phi_{\delta'}} & Ll;Lr\ar[d]^{Ls_{1};Ls_{2}}\\
l';r'\ar[r]_{t_{1};t_{2}}\ar@{..>}[ur]|-{\gamma_{1};\gamma_{2}} & a;b &  &  & Ll';Lr'\ar[r]_{Lt_{1};Lt_{2}}\ar[ur]|-{L\gamma_{1};L\gamma_{2}} & La;Lb
}
\label{deflambda}
\end{equation}
composing to the identity, and this implies the right diagram commutes
by naturality of comultiplication (with $\zeta$ taken to be the identity).
Trivially, we take each unit $\lambda_{X}\colon L\left(1_{X}\right)\to1_{X}$
to be the component of $\Lambda$ at $\textnormal{id}_{1_{X}}$.

To see that the family $\varphi$ satisfies naturality of the constraints
suppose that we are given a diagram as on the left below with the
horizontal paths composing to identities
\[
\xymatrix{a;b\myar{\delta}{r}\ar[d]_{\alpha;\beta} & l;r\myar{s_{1};s_{2}}{r}\ar@{..>}[d]^{\gamma_{1};\gamma_{2}} & a;b\ar[d]^{\alpha;\beta} &  & L\left(a;b\right)\myar{\Phi_{\delta}}{r}\ar[d]_{L\left(\alpha;\beta\right)} & Ll;Lr\myar{Ls_{1};Ls_{2}}{r}\ar[d]^{L\gamma_{1};L\gamma_{2}} & La;Lb\ar[d]^{L\alpha;L\beta}\\
a';b'\myard{\delta'}{r} & l';r'\myard{s_{1}';s_{2}'}{r} & a';b' &  & L\left(a';b'\right)\myard{\Phi_{\delta'}}{r} & Ll';Lr'\myard{Ls_{1}';Ls_{2}'}{r} & La';Lb'
}
\]
and note that the right diagram commutes by naturality of comultiplication.

Before checking associativity we first note that given any generics
$\delta_{1}',\delta_{2}',\delta_{3}'$ and $\delta_{4}'$ in $\mathscr{A}$
such that $(1)$ commutes below,
\[
\xymatrix@=1em{c\myar{\zeta^{-1}}{rr}\ar[d]_{\delta_{3}}\ar@{}[rrd]|-{(5)} &  & c'\ar[d]_{\delta_{3}'}\ar@{=}[rr]\ar@{}[rrdd]|-{(1)} &  & c'\ar[d]^{\delta_{1}'}\myar{\zeta}{rr}\ar@{}[rrd]|-{(2)} &  & c\ar[d]^{\delta_{1}}\\
l;k\ar[d]_{l;\delta_{4}}\ar[rr]|-{\alpha;\beta}\ar@{}[rrd]|-{(6)} &  & l';k'\myard{l';\delta_{4}'}{d} &  & h';r'\ar[d]^{\delta_{2}';r'}\ar[rr]|-{\zeta_{1};\zeta_{2}}\ar@{}[rrd]|-{(3)} &  & h;r\ar[d]^{\delta_{2};r}\\
l;\left(m;r\right)\myard{\phi_{1}^{-1};\left(\phi_{2}^{-1};\zeta_{2}^{-1}\right)}{rr} &  & l';\left(m';r'\right)\myard{\textnormal{assoc}}{rr}\ar@/_{1pc}/[rrd]|-{\phi_{1};\left(\phi_{2};\zeta_{2}\right)} &  & \left(l';m'\right);r'\myard{\left(\phi_{1};\phi_{2}\right);\zeta_{2}}{rr}\ar@{}[d]|-{(4)} &  & \left(l;m\right);r\\
 &  &  &  & l;\left(m;r\right)\ar@/_{1pc}/[rru]|-{\textnormal{assoc}}
}
\]
we can construct regions (2) and (3) as on the right above, where
$\delta_{1}$ and $\delta_{2}$ lie in $\Delta_{2}$. By naturality
of the associator (4) commutes. Then since our given class of generics
is coherent, we can find a $\delta_{3}$ and $\delta_{4}$ in $\Delta_{2}$
such that the outside diagram commutes above. By genericity of $\delta_{3}$
we then have induced 2-cells $\alpha$ and $\beta$ such that (5)
and (6) commute (invertible as $\delta_{3}'$ is also generic). Now,
by associativity of comultiplication the commutativity of the outside
diagram is respected by the transformation $\delta\mapsto\Phi_{\delta}$,
and this is equivalent to the commutativity of (1) being respected
as the pasting with (2),(3),(4),(5) and (6) may be undone.

Now, to see that the family $\varphi$ satisfies associativity of
the constraints consider the outside diagram of
\[
\xymatrix{\left(a;b\right);c\myar{\delta_{1}}{r}\ar[rd]_{\delta_{3}}\ar[dddd]_{\textnormal{assoc}} & h;r\myar{s_{1};s_{2}}{r}\ar[rd]_{\delta_{2};r} & \left(a;b\right);c\myar{\delta_{5};c}{r}\ar@{}[d]|-{(1)} & \left(f;g\right);c\myar{\left(t_{1};t_{2}\right);c}{r} & \left(a;b\right);c\ar[dddd]^{\textnormal{assoc}}\\
 & l;k\ar[rd]_{l;\delta_{4}}\ar@{..>}[ddd]_{\gamma_{1};\gamma_{2}}\ar@{}[r]|-{(3)} & \left(l;m\right);r\ar[d]|-{\textnormal{assoc}}\ar[ru]_{\left(\xi_{1};\xi_{2}\right);s_{2}}\\
 & \ar@{}[l]|-{(5)} & l;\left(m;r\right)\ar[rd]^{\xi_{1};\left(\xi_{2};s_{2}\right)}\ar@{..>}[d]_{\gamma_{1};\left(\alpha;\beta\right)} &  & \ar@{}[llu]|-{(4)}\\
 & \ar@{}[r]|-{(6)} & \widetilde{l};\left(\widetilde{m};\widetilde{r}\right)\ar[rd]^{p_{1};\left(\zeta_{1};\zeta_{2}\right)}\ar@{}[d]|-{(2)}\ar@{}[r]|-{(7)} & f;\left(g;c\right)\ar[rd]^{t_{1};\left(t_{2};c\right)}\\
a;\left(b;c\right)\myard{\delta_{6}}{r} & \widetilde{l};\widetilde{k}\myard{p_{1};p_{2}}{r}\ar[ru]^{\widetilde{l};\delta_{7}} & a;\left(b;c\right)\myard{a;\delta_{8}}{r} & a;\left(u;v\right)\myard{a;\left(q_{1};q_{2}\right)}{r} & a;\left(b;c\right)
}
\]
where the appropriate horizontal composites are identity 2-cells.
We first factor $\delta_{5}s_{1}$ through a generic $\delta_{2}$
to recover 2-cells $\xi_{1}$ and $\xi_{2}$ and the commuting region
(1). Similarly, we create the region (2). Now take $\delta_{3}$ and
$\delta_{4}$ to be generics such that region (3) commutes, which
exist by Lemma \ref{cohassoc}. We then note that region (4) commutes
by naturality of the associator in $\mathscr{A}$. Finally, note that
we have an induced $\left(\gamma_{1};\gamma_{2}\right)$ by genericity
of $\delta_{3}$, and thus $\delta_{7}\gamma_{2}$ yields an induced
$\left(\alpha;\beta\right)$ through the generic $\delta_{4}$.

We have now constructed the above diagram and shown each region commutes;
all that remains is to notice in the corresponding diagram below
\[
\xymatrix{L\left(\left(a;b\right);c\right)\myar{\Phi_{\delta_{1}}}{r}\ar[rd]_{\Phi_{\delta_{3}}}\ar[dddd]_{L\left(\textnormal{assoc}\right)} & Lh;Lr\myar{Ls_{1};Ls_{2}}{r}\ar[rd]_{\Phi_{\delta_{2}};Lr} & L\left(a;b\right);Lc\myar{\Phi_{\delta_{5}};c}{r}\ar@{}[d]|-{(1)} & \left(Lf;Lg\right);Lc\myar{\left(Lt_{1};Lt_{2}\right);Lc}{r} & \left(La;Lb\right);Lc\ar[dddd]^{\textnormal{assoc}}\\
 & Ll;Lk\ar[rd]_{Ll;\Phi_{\delta_{4}}}\ar[ddd]_{L\gamma_{1};L\gamma_{2}}\ar@{}[r]|-{(3)} & \left(Ll;Lm\right);Lr\ar[d]|-{\textnormal{assoc}}\ar[ru]_{\left(L\xi_{1};L\xi_{2}\right);Ls_{2}}\\
\; & \ar@{}[l]|-{(5)} & Ll;\left(Lm;Lr\right)\ar[rd]^{L\xi_{1};\left(L\xi_{2};Ls_{2}\right)}\ar[d]_{L\gamma_{1};\left(L\alpha;L\beta\right)} &  & \ar@{}[llu]|-{(4)}\\
 & \ar@{}[r]|-{(6)} & L\widetilde{l};L\left(\widetilde{m};\widetilde{r}\right)\ar[rd]^{Lp_{1};\left(L\zeta_{1};L\zeta_{2}\right)}\ar@{}[d]|-{(2)}\ar@{}[r]|-{(7)} & Lf;\left(Lg;Lc\right)\ar[rd]^{Lt_{1};\left(Lt_{2};Lc\right)}\\
L\left(a;\left(b;c\right)\right)\myard{\Phi_{\delta_{6}}}{r} & L\widetilde{l};L\widetilde{k}\myard{Lp_{1};Lp_{2}}{r}\ar[ru]^{L\widetilde{l};\Phi_{\delta_{7}}} & La;L\left(b;c\right)\myard{La;\Phi_{\delta_{8}}}{r} & La;\left(Lu;Lv\right)\myard{La;\left(Lq_{1};Lq_{2}\right)}{r} & La;\left(Lb;Lc\right)
}
\]
naturality of comultiplication implies (1), (2), (5) and (6) commute;
associativity of comultiplication implies (3) commutes; naturality
of the associators in $\mathscr{C}$ implies (4) commutes, and (7)
commutes as $L$ is locally a functor.

Before checking the unitary axioms on $\lambda$ we note that given
a generic $\delta'$ and augmentation $\epsilon'$ composable as in
the middle diagram below
\[
\xymatrix@=1em{ &  &  & n;c\ar@/^{1pc}/[rrrdd]^{\epsilon;c}\ar@{..>}[d]^{u_{1};u_{2}}\\
 &  &  & n';c'\ar[rd]^{\epsilon';c'}\\
c\myard{\zeta}{rr}\ar@/^{1pc}/[rrruu]^{\delta} &  & c'\myard{\textnormal{unitor}}{rr}\ar[ur]^{\delta'} &  & 1_{X};c'\myard{1_{X};\zeta^{-1}}{rr} &  & 1_{X};c
}
\]
we have an isomorphism $\zeta\colon c\to c'$ by axiom (1) of a coherent
class. By axiom (5) we then have a $\delta$ and $\epsilon$ in the
coherent class such that the outside diagram commutes. It follows
from genericity of $\delta$ that we have an induced isomorphism $u_{1};u_{2}$
such that the above diagram commutes. As the commutativity of the
outside diagram is respected by assumption, and the commutativity
of the left and right regions is respected by naturality of comultiplication
and augmentations respectively (and the pasting with these regions
can be undone), it follows that the commutativity of the middle diagram
is respected.

Now, to see the left unit axiom on $\lambda$ is satisfied note that
given any commuting diagram as on the left below
\[
\xymatrix{1_{X};c\myar{\textnormal{unitor}}{r}\ar[d]_{\delta} & c\myar{\textnormal{unitor}}{r}\ar[d]_{\delta'} & 1_{X};c & L\left(1_{X};c\right)\myar{L\left(\textnormal{unitor}\right)}{r}\ar[d]_{\Phi_{\delta}} & Lc\myar{\textnormal{unitor}}{r}\ar[d]_{\Phi_{\delta'}} & 1_{X};Lc\\
l;r\myard{s_{1};s_{2}}{r} & l';c\ar[ur]|-{\epsilon;c}\myard{\epsilon;c}{r} & 1_{X};c\ar[u]_{\textnormal{id}} & Ll;Lr\myard{Ls_{1};Ls_{2}}{r} & Ll';Lc\ar[ur]|-{\Lambda_{\epsilon};Lc}\myard{L\epsilon;Lc}{r} & L1_{X};Lc\ar[u]_{\Lambda_{1_{X}};Lc}
}
\]
we get a commuting diagram as on the right above by naturality of
comultiplication, the left counit axiom, and naturality of counits
(the bottom composite in this diagram is a $\varphi$ followed by
a $\lambda$). The right unitary axiom is similar.

Finally, note that the composite assignment 
\[
\left(1\right)\mapsto\left(2\right)\mapsto\left(1\right)
\]
is the identity, since with $\Phi$ defined as in \eqref{defPhi},
the oplax constraint cells as recovered by \eqref{defphi}, given
by the family of constraints
\[
\xymatrix{L\left(a;b\right)\myar{L\delta}{r} & L\left(l;r\right)\myar{\varphi_{l,r}}{r} & Ll;Lr\myar{Ls_{1};Ls_{2}}{r} & La;Lb}
\]
are clearly equal to $\varphi_{a,b}$ by naturality. Moreover, the
composite assignment
\[
\left(2\right)\mapsto\left(1\right)\mapsto\left(2\right)
\]
is the identity, since with $\varphi$ defined as by \eqref{defphi},
the comultiplication cells $\Phi$ at an arbitrary generic $\widetilde{\delta}\in\Delta_{2}$
are given by the composite in the top line on the left below
\begin{equation}
\xymatrix@=1em{ &  &  &  & \;\ar@{}[d]|-{:=}\\
Lc\myar{L\widetilde{\delta}}{rr}\ar[rrd]_{\Phi_{\widetilde{\delta}}} &  & L\left(\widetilde{l};\widetilde{r}\right)\myar{\Phi_{\delta}}{rr}\ar@{}[d]|-{\left(1\right)}\ar@/^{2pc}/[rrrr]^{\varphi_{\tilde{l},\tilde{r}}} &  & Ll;Lr\myar{Ls_{1};Ls_{2}}{rr} &  & L\widetilde{l};L\widetilde{r} &  & l;r\ar[rd]^{s_{1};s_{2}}\\
 &  & L\widetilde{l};L\widetilde{r}\ar[rru]_{L\widetilde{\delta_{1}};L\widetilde{\delta_{2}}} &  &  &  &  & \widetilde{l};\widetilde{r}\myard{\textnormal{id}}{rr}\myar{\delta}{ru} &  & \widetilde{l};\widetilde{r}
}
\label{212}
\end{equation}
where $\delta\in\Delta_{2}$ is a generic and the right diagram commutes.
Then we note that
\[
\xymatrix{c\myar{\widetilde{\delta}}{r}\ar[d]_{\widetilde{\delta}} & \widetilde{l};\widetilde{r}\ar@{..>}[d]^{\widetilde{\delta_{1}};\widetilde{\delta_{2}}} &  & c\myar{\widetilde{\delta}}{r}\ar[d]_{\widetilde{\delta}} & \widetilde{l};\widetilde{r}\ar[d]^{\textnormal{id};\textnormal{id}}\ar@{..>}[dl]|-{\textnormal{id};\textnormal{id}} &  & c\myar{\widetilde{\delta}}{r}\ar[d]_{\widetilde{\delta}} & \widetilde{l};\widetilde{r}\ar@{..>}[ld]|-{\textnormal{id};\textnormal{id}}\ar[d]^{s_{1}\widetilde{\delta_{1}};s_{2}\widetilde{\delta_{2}}}\\
\widetilde{l};\widetilde{r}\myard{\delta}{r} & l;r &  & \widetilde{l};\widetilde{r}\myard{\textnormal{id};\textnormal{id}}{r} & \widetilde{l};\widetilde{r} &  & \widetilde{l};\widetilde{r}\myard{\textnormal{id};\textnormal{id}}{r} & \widetilde{l};\widetilde{r}
}
\]
 we have an induced $\widetilde{\delta_{1}};\widetilde{\delta_{2}}$
rendering commutative the left diagram above by genericity of $\widetilde{\delta}$,
the middle diagram shows that the induced diagonal is necessarily
a pair of identities (by component-wise commutativity of the bottom
triangle), and whiskering the left diagram with $s_{1};s_{2}$ gives
the right diagram, where as we have noted the induced diagonal making
the diagram commute is a pair of identities. Consequently, $s_{1}\widetilde{\delta_{1}}$
and $s_{2}\widetilde{\delta_{2}}$ are identities. We then note that
in diagram \ref{212} the region (1) commutes by naturality of comultiplication,
and applying local functoriality of $L$ we then see the given composite
is $\Phi_{\tilde{\delta}}$ as required. 

The bijection of the nullary data may be similarly proven using the
respective naturality properties, and so we omit the details.

\noun{Bijection With Oplax Natural Transformations.} As the the data
of (1) and (2) is the same, we need only check that the coherence
conditions correspond.

$\left(1\right)\Longrightarrow\left(2\right)\colon$ Suppose we are
given an oplax natural transformation $\vartheta\colon L\to K$ in
the usual sense. Then by the definition of $\Phi$ at a $\delta\in\Delta_{2}$
we have
\[
\xymatrix{ &  &  &  &  & \;\ar@{}[d]|-{\Downarrow L\delta}\\
LX\myar{Lc}{rr}\ar[d]_{\vartheta_{X}}\ar@/_{1pc}/[rd]^{Ll} & \ar@{}[d]|-{\Downarrow\Phi_{\delta}} & LZ\ar[d]^{\vartheta_{Z}} &  & LX\ar[d]_{\vartheta_{X}}\ar@/_{1pc}/[rd]^{Ll}\ar[rr]|-{L\left(l;r\right)}\ar@/^{2pc}/[rr]^{L\left(c\right)} & \ar@{}[d]|-{\Downarrow\varphi_{l,r}} & LZ\ar[d]^{\vartheta_{Z}}\\
KX\ar@/_{1pc}/[rd]_{Kl}\ar@{}[dr]|-{\Downarrow\vartheta_{l}} & LY\ar@/_{1pc}/[ru]^{Lr}\ar[d]|-{\theta_{Y}} & KZ\ar@{}[dl]|-{\Downarrow\vartheta_{r}} & = & KX\ar@/_{1pc}/[rd]_{Kl}\ar@{}[dr]|-{\Downarrow\vartheta_{l}} & LY\ar@/_{1pc}/[ru]^{Lr}\ar[d]|-{\theta_{Y}} & KZ\ar@{}[dl]|-{\Downarrow\vartheta_{r}}\\
 & KY\ar@/_{1pc}/[ur]_{Kr} &  &  &  & KY\ar@/_{1pc}/[ur]_{Kr}
}
\]
which by compatibility with composition is
\[
\xymatrix{ & \;\ar@{}[d]|-{\Downarrow L\delta} &  &  &  & \;\ar@{}[d]|-{\overset{\;}{\Downarrow\vartheta_{c}}}\\
LX\ar[d]_{\vartheta_{X}}\ar[rr]|-{L\left(l;r\right)}\ar@/^{2pc}/[rr]^{L\left(c\right)}\ar@{}[drr]|-{\Downarrow\vartheta_{l;r}} &  & LZ\ar[d]^{\vartheta_{Z}} &  & LX\ar[d]_{\vartheta_{X}}\ar@/^{2pc}/[rr]^{L\left(l;r\right)} & \ar@{}[d]|-{\overset{\;}{\Downarrow K\delta}} & LZ\ar[d]^{\vartheta_{Z}}\\
KX\ar@/_{1pc}/[rd]_{Kl}\ar[rr]|-{K\left(l;r\right)} & \ar@{}[d]|-{\Downarrow\psi_{l,r}} & KZ & = & KX\ar@/_{1pc}/[rd]_{Kl}\ar[rr]|-{K\left(l;r\right)}\ar@/^{2pc}/[rr]^{K\left(c\right)} & \ar@{}[d]|-{\Downarrow\psi_{l,r}} & KZ\\
 & KY\ar@/_{1pc}/[ur]_{Kr} &  &  &  & KY\ar@/_{1pc}/[ur]_{Kr}
}
\]
and by definition of $\Psi$ this gives the required coherence condition.
We omit the nullary version.

$\left(2\right)\Longrightarrow\left(1\right)\colon$ Suppose we are
given the data of (2) subject to the coherence conditions of (2).
Then by the definition of the constraint data $\varphi$ we have 
\[
\xymatrix{ & \;\ar@{}[d]|-{\Downarrow\varphi_{f,g}} &  &  &  & \;\ar@{}[d]|-{\Downarrow\Phi_{\delta}}\\
LX\myar{Lf}{r}\ar@/^{2.5pc}/[rr]^{L\left(f;g\right)}\ar[d]_{\vartheta_{X}}\ar@{}[rd]|-{\Downarrow\vartheta_{f}} & LY\myar{Lg}{r}\ar[d]|-{\vartheta_{Y}}\ar@{}[rd]|-{\Downarrow\vartheta_{g}} & LZ\ar[d]^{\vartheta_{Z}} & = & LX\ar@/_{0.4pc}/[r]_{Lf}\ar@/^{2.5pc}/[rr]^{L\left(f;g\right)}\ar[d]_{\vartheta_{X}}\ar@{}[rd]|-{\stackrel{\;}{\Downarrow\vartheta_{f}}}\ar@/^{0.4pc}/[r]^{Ll}\ar@{}[r]|-{\Downarrow Ls_{1}} & LY\ar[d]|-{\vartheta_{Y}}\ar@{}[rd]|-{\stackrel{\;}{\Downarrow\vartheta_{g}}}\ar@/^{0.4pc}/[r]^{Lr}\ar@/_{0.4pc}/[r]_{Lg}\ar@{}[r]|-{\Downarrow Ls_{2}} & LZ\ar[d]^{\vartheta_{Z}}\\
KX\myard{Kf}{r} & KY\myard{Kg}{r} & KZ &  & KX\myard{Kf}{r} & KY\myard{Kg}{r} & KZ
}
\]
and so applying naturality of $\vartheta$, this is equal to the left
below
\[
\xymatrix{ & \;\ar@{}[d]|-{\Downarrow\Phi_{\delta}} &  &  &  & \ar@{}[d]|-{\Downarrow\vartheta_{f;g}}\\
LX\myar{Ll}{r}\ar[d]_{\vartheta_{X}}\ar@{}[rd]|-{\Downarrow\vartheta_{l}}\ar@/^{2.5pc}/[rr]^{L\left(f;g\right)} & LY\myar{Lr}{r}\ar[d]|-{\vartheta_{Y}}\ar@{}[rd]|-{\Downarrow\vartheta_{r}} & LZ\ar[d]^{\vartheta_{Z}} & = & LX\ar@/^{2.5pc}/[rr]^{L\left(f;g\right)}\ar[d]_{\vartheta_{X}} & \;\ar@{}[d]|-{\Downarrow\Psi_{\delta}} & LZ\ar[d]^{\vartheta_{Z}}\\
KX\ar@/_{0.4pc}/[r]_{Kf}\ar@/^{0.4pc}/[r]^{Kl}\ar@{}[r]|-{\Downarrow Ks_{1}} & KY\ar@/^{0.4pc}/[r]^{Kr}\ar@/_{0.4pc}/[r]_{Kg}\ar@{}[r]|-{\Downarrow Ks_{2}} & KZ &  & KX\ar@/_{0.4pc}/[r]_{Kf}\ar@/^{0.4pc}/[r]^{Kl}\ar@{}[r]|-{\Downarrow Ks_{1}}\ar@/^{2.5pc}/[rr]^{K\left(f;g\right)} & KY\ar@/^{0.4pc}/[r]^{Kr}\ar@/_{0.4pc}/[r]_{Kg}\ar@{}[r]|-{\Downarrow Ks_{2}} & KZ
}
\]
which by the assumed coherence axiom is the right above. Applying
the definition of $\psi$, we recover the compatibility of an oplax
natural transformation with composition. Again, we will omit the analogous
nullary condition.

\noun{Bijection With Icons. }This trivially follows taking each $\vartheta_{X}$
to be an identity 1-cell in the above bijection.
\end{proof}
\begin{rem}
Notice that in Theorem \ref{generalcoherence}, giving binary oplax
constraint cells
\[
\varphi_{l,r}\colon L\left(l;r\right)\to Ll;Lr
\]
for generics $\delta\colon c\to l;r$ in $\Delta_{2}$ completely
determines arbitrary oplax constraint cells
\[
\varphi_{a,b}\colon L\left(a;b\right)\to La;Lb
\]
This is since these $\varphi_{l,r}$ suffice to construct each $\Phi_{\delta}$.
Hence this theorem provides a reduction in the data of an oplax functor
when the domain bicategory $\mathscr{A}$ is generic. 
\end{rem}

\begin{rem}
Given a family of hom-categories $\mathscr{A}_{X,Y}$, sets $\mathfrak{M}_{c}^{X,Y,Z}$,
and natural isomorphisms 
\[
\mathscr{A}_{X,Z}\left(c,a;b\right)\cong\sum_{m\in\mathfrak{M}_{c}^{X,Y,Z}}\mathscr{A}_{X,Y}\left(l_{m},a\right)\times\mathscr{A}_{Y,Z}\left(r_{m},b\right)
\]
for all $X,Y,Z$ and $c$, the formal composite $a;b$ is essentially
uniquely determined (by essential uniqueness of representing objects).

Given a complete class of generics $\Delta_{2}$ equipped with their
universal properties, one may recover the above by taking $\mathfrak{M}_{c}^{X,Y,Z}$
to be the set of equivalence classes of generics $\delta\colon c\to l;r$.
It follows that composition in the bicategory is essentially uniquely
determined by the generics.
\end{rem}

\section{Consequences and examples\label{consequences}}

In this section we discuss some of the main examples of Theorem \ref{generalcoherence}.
Viewing monoidal categories as one-object bicategories, we first consider
the case where $\mathscr{A}$ is a cartesian monoidal category, giving
a simple and informative example of this situation. We then go on
to consider more complicated examples, namely where $\mathscr{A}$
is the bicategory of spans or the bicategory of polynomials with cartesian
2-cells.

For completeness, we also discuss the case where $\mathscr{A}$ is
the category of finite sets and bijections with the disjoint union
monoidal structure, but will omit some details as this is a rather
trivial example.

\subsection{Cartesian monoidal categories}

Given a category $\mathcal{E}$ with finite products, one may construct
the cartesian monoidal category $\left(\mathcal{E},\times,\mathbf{1}\right)$
where the tensor product is the cartesian product and the unit is
the terminal object. Clearly this monoidal category is generic, as
\[
\mathcal{E}\left(T,-\times-\right)\colon\mathcal{E}\times\mathcal{E}\to\mathbf{Set}
\]
is representable (no coproducts are necessary). Now, seen as a one
object bicategory, the generics are the diagonal morphisms $\delta_{T}$
in $\mathcal{E}$ of the form
\[
\xymatrix@=1em{ &  & T\ar@{..>}[dd]|-{\delta_{T}}\ar[rrdd]^{\textnormal{id}}\ar[lldd]_{\textnormal{id}}\\
\\
T &  & T\times T\myard{\pi_{2}}{rr}\myar{\pi_{1}}{ll} &  & T
}
\]
and so we take $\Delta_{2}$ to be the class of diagonals $\delta_{T}\colon T\to T\times T$
for each $T\in\mathcal{E}$. Trivially, we take the augmentations
as the unique maps into the terminal object from each object $T\in\mathcal{E}$.
Applying Theorem \ref{generalcoherence} in this case then makes it
clear why we may say the data of this theorem is analogous to the
data of a comonad; indeed, we have the following.
\begin{cor}
Let $\mathcal{E}$ be a category with finite products and let $\left(\mathcal{C},\otimes,I\right)$
be a monoidal category. Denote by $\left(\mathcal{E},\times,\mathbf{1}\right)$
the category $\mathcal{E}$ equipped with the cartesian monoidal structure.
Then to give an oplax monoidal functor 
\[
L\colon\left(\mathcal{E},\times,\mathbf{1}\right)\to\left(\mathcal{C},\otimes,I\right)
\]
 is to give a functor $L\colon\mathcal{E}\to\mathcal{C}$ with comultiplication
and counit maps 
\[
\Phi_{T}\colon L\left(T\right)\to L\left(T\right)\otimes L\left(T\right),\qquad\Lambda_{T}\colon L\left(T\right)\to I
\]
for every $T\in\mathcal{E}$, such that for every $T\in\mathcal{E}$
the diagrams
\[
\xymatrix@=1em{ & LT\otimes LT\ar[dr]^{LT;\Lambda_{T}} &  &  &  & LT\otimes LT\ar[dr]^{\Lambda_{T};LT}\\
LT\myard{\textnormal{unitor}}{rr}\ar[ur]^{\Phi_{T}} &  & LT\otimes I &  & LT\myard{\textnormal{unitor}}{rr}\ar[ur]^{\Phi_{T}} &  & I\otimes LT
}
\]
commute, the diagrams
\[
\xymatrix@=1em{LT\ar[d]_{\Phi_{T}}\ar@{=}[rr] &  & LT\ar[d]^{\Phi_{T}}\\
LT\otimes LT\myard{LT\otimes\Phi_{T}}{d} &  & LT\otimes LT\ar[d]^{\Phi_{T}\otimes LT}\\
LT\otimes\left(LT\otimes LT\right)\myard{\textnormal{assoc}}{rr} &  & \left(LT\otimes LT\right)\otimes LT
}
\]
commute, and all morphisms $f\colon T\to T'$ in $\mathcal{E}$ render
commutative
\[
\xymatrix@=1.5em{L\left(T\right)\myar{Lf}{rr}\ar[rd]_{\Lambda_{T}} &  & L\left(T'\right)\ar[ld]^{\Lambda_{T'}} &  & LT\myar{\Phi_{T}}{rr}\ar[d]_{Lf} &  & LT\otimes LT\ar[d]^{Lf\otimes Lf}\\
 & 1_{X} &  &  & LT'\myard{\Phi_{T'}}{rr} &  & LT'\otimes LT'
}
\]
\end{cor}

The unitary and associativity conditions above ask that $L$ sends
each $T\in\mathcal{E}$ to a comonoid $\left(LT,\Phi_{T},\Lambda_{T}\right)$
in $\left(\mathcal{C},\otimes,I\right)$, and the last two conditions
ask that morphisms in $\mathcal{E}$ are sent to morphisms of comonoids.
Hence this may be simply stated as follows.
\begin{cor}
Let $\mathbf{Comon}\left(\mathcal{C},\otimes,I\right)$ be the category
of comonoids in the monoidal category $\left(\mathcal{C},\otimes,I\right)$.
Then oplax monoidal functors $\left(\mathcal{E},\times,\mathbf{1}\right)\to\left(\mathcal{C},\otimes,I\right)$
are in bijection with functors $\mathcal{E}\to\mathbf{Comon}\left(\mathcal{C},\otimes,I\right)$.
\end{cor}

\subsection{Bicategories of spans}

Given a category $\mathcal{E}$ with pullbacks, one may form the bicategory
of spans in $\mathcal{E}$ denoted $\mathbf{Span}\left(\mathcal{E}\right)$
with objects those of $\mathcal{E}$, 1-cells given by spans
\[
\xymatrix@=1em{ & T\ar[rd]^{t}\ar[ld]_{s}\\
X &  & Z
}
\]
denoted $\left(s,t\right)$, 2-cells given by morphisms $f$ rendering
commutative diagrams as on the left below
\[
\xymatrix@=1em{ & K\ar[rd]^{b}\ar[ld]_{a}\ar[dd]|-{f} &  &  &  &  &  &  & M\ar[rd]^{\pi_{2}}\ar[ld]_{\pi_{1}}\ar@{}[dd]|-{\underset{\;}{\textnormal{pb}}}\\
X &  & Y &  &  &  &  & R\ar[rd]^{v}\ar[ld]_{u} &  & S\ar[rd]^{q}\ar[ld]_{p}\\
 & R\ar[ur]_{v}\ar[lu]^{u} &  &  &  &  & X &  & Y &  & Z
}
\]
and composition of 1-cells given by forming the pullback as on the
right above \cite{ben1967}. 

The reader will then notice that by the universal property of pullback,
giving a morphism of spans $\left(s,t\right)\to\left(u,v\right);\left(p,q\right)$
as on the left below
\[
\xymatrix@=1em{ &  & T\ar@{..>}[d]\ar@/^{1pc}/[rrddd]^{t}\ar@/_{1pc}/[llddd]_{s} &  &  &  &  &  &  & T\ar[ddd]|-{h}\ar@/^{1pc}/[rrddd]^{t}\ar@/_{1pc}/[llddd]_{s}\ar@{..>}[ddl]\ar@{..>}[ddr]\\
 &  & M\ar[rd]^{\pi_{2}}\ar[ld]_{\pi_{1}}\ar@{}[dd]|-{\underset{\;}{\textnormal{pb}}} &  &  & \ar@{}[rd]|-{\sim}\\
 & R\ar[rd]^{v}\ar[ld]_{u} &  & S\ar[rd]^{q}\ar[ld]_{p} &  &  & \; &  & R\ar[rd]^{v}\ar[ld]_{u} &  & S\ar[rd]^{q}\ar[ld]_{p}\\
X &  & Y &  & Z &  &  & X &  & Y &  & Z
}
\]
is to give a morphism $h\colon T\to Y$ as well as pair of morphisms
of spans as on the right above such that each region in the diagram
commutes. Therefore
\[
\mathbf{Span}\left(\mathcal{E}\right)_{X,Z}\left(\left(s,t\right),\left(u,v\right);\left(p;q\right)\right)
\]
is isomorphic to
\begin{equation}
\sum_{h\colon H\to Y}\mathbf{Span}\left(\mathcal{E}\right)_{X,Y}\left(\left(s,h\right),\left(u,v\right)\right)\times\mathbf{Span}\left(\mathcal{E}\right)_{Y,Z}\left(\left(h,t\right),\left(p,q\right)\right)\label{spanfamrep}
\end{equation}
and so the bicategory of spans is generic. Our class of generics $\Delta_{2}$
consists of, for each diagram
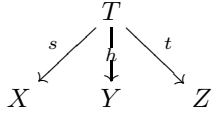
\[
\xymatrix@=0.7em{ &  & T\ar[rrdd]^{t}\ar[lldd]_{s}\ar[dd]|-{h}\\
\\
X &  & Y &  & Z
}
\]
in $\mathcal{E}$, the morphisms of spans $\delta_{s,h,t}\colon\left(s,t\right)\to\left(s,h\right);\left(h,t\right)$
corresponding to
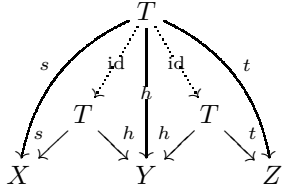
\[
\xymatrix@=1em{ &  & T\ar[ddd]|-{h}\ar@/^{1pc}/[rrddd]^{t}\ar@/_{1pc}/[llddd]_{s}\ar@{..>}|-{\textnormal{id}}[ddl]\ar@{..>}[ddr]|-{\textnormal{id}}\\
\\
 & T\ar[rd]^{h}\ar[ld]_{s} &  & T\ar[rd]^{t}\ar[ld]_{h}\\
X &  & Y &  & Z
}
\]
under this bijection. Our augmentations are the morphisms of spans
as below for each morphism $h$ in $\mathcal{E}$
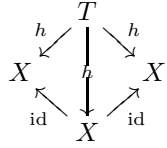
\[
\xymatrix@=1em{ & T\ar[rd]^{h}\ar[ld]_{h}\ar[dd]|-{h}\\
X &  & X\\
 & X\ar[ur]_{\textnormal{id}}\ar[lu]^{\textnormal{id}}
}
\]
and will be denoted by $\epsilon_{h}$. Thus, applying Theorem \ref{generalcoherence}
we have the following.
\begin{cor}
Let $\mathcal{E}$ be a category with pullbacks and denote by $\mathbf{Span}\left(\mathcal{E}\right)$
the bicategory of spans in $\mathcal{E}$. Let $\mathscr{C}$ be a
bicategory. Then to give an oplax functor
\[
L\colon\mathbf{Span}\left(\mathcal{E}\right)\to\mathscr{C}
\]
 is to give a locally defined functor
\[
L_{X,Y}\colon\mathbf{Span}\left(\mathcal{E}\right)_{X,Y}\to\mathscr{C}_{LX,LY},\qquad X,Y\in\mathcal{E}
\]
with comultiplication and counit maps 
\[
\Phi_{s,h,t}\colon L\left(s,t\right)\to L\left(s,h\right);L\left(h,t\right),\qquad\Lambda_{h}\colon L\left(h,h\right)\to1_{LX}
\]
for every respective diagram in $\mathcal{E}$ 
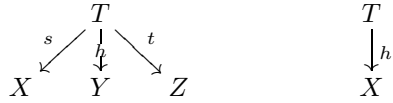
\[
\xymatrix@=1.5em{ & T\ar[rd]^{t}\ar[ld]_{s}\ar[d]|-{h} &  &  &  & T\ar[d]^{h}\\
X & Y & Z &  &  & X
}
\]
such that:
\begin{enumerate}
\item for any triple of morphisms of spans as below
\[
\xymatrix@=1.5em{ & R\ar[dl]_{u}\ar[dr]^{v}\ar[dd]|-{f} &  &  & R\ar[dl]_{u}\ar[dr]^{k}\ar[dd]|-{f} &  &  & R\ar[dl]_{k}\ar[dr]^{v}\ar[dd]|-{f}\\
X &  & Z & X &  & Y & Y &  & Z\\
 & T\ar[ul]^{s}\ar[ru]_{t} &  &  & T\ar[ul]^{s}\ar[ru]_{h} &  &  & T\ar[ul]^{h}\ar[ru]_{t}
}
\]
we have the commuting diagram
\[
\xymatrix{L\left(u,v\right)\ar[d]_{Lf}\myar{\Phi_{u,k,v}}{rr} &  & L\left(u,k\right);L\left(k,v\right)\ar[d]^{Lf;Lf}\\
L\left(s,t\right)\myard{\Phi_{s,h,t}}{rr} &  & L\left(s,h\right);L\left(h,t\right)
}
\]
\item for any morphism of spans as on the left below
\[
\xymatrix@=1.5em{ & M\ar[dl]_{p}\ar[dr]^{p}\ar[d]|-{f} &  &  & L\left(p,p\right)\ar[rr]^{Lf}\ar[rd]_{\Lambda_{p}} &  & L\left(q,q\right)\ar[ld]^{\Lambda_{q}}\\
X & N\ar[l]^{q}\ar[r]_{q} & X &  &  & 1_{LX}
}
\]
the diagram on the right above commutes;
\item for all diagrams of the form
\[
\xymatrix@=1.5em{ &  & T\ar[dll]_{s}\ar[drr]^{t}\ar[dl]^{h}\ar[dr]_{k}\\
W & X &  & Y & Z
}
\]
in $\mathcal{E}$, we have the commuting diagram
\[
\xymatrix{L\left(s,t\right)\ar[d]_{\Phi_{s,h,t}}\ar@{=}[rr] &  & L\left(s,t\right)\ar[d]^{\Phi_{s,k,t}}\\
L\left(s,h\right);L\left(h,t\right)\myard{L\left(s;h\right);\Phi_{h,k,t}}{d} &  & L\left(s,k\right);L\left(k,t\right)\ar[d]^{\Phi_{s,h,k};L\left(k;t\right)}\\
L\left(s,h\right);\left(L\left(h,k\right);L\left(k,t\right)\right)\myard{\textnormal{assoc}}{rr} &  & \left(L\left(s,h\right);L\left(h,k\right)\right);L\left(k,t\right)
}
\]
\item for all spans $\left(s,t\right)$ we have the commuting diagrams
\[
\xymatrix@=0.5em{ & L\left(s,s\right);L\left(s,t\right)\ar[rdd]^{\Lambda_{s};L\left(s,t\right)} &  &  &  & L\left(s,t\right);L\left(t,t\right)\ar[rdd]^{L\left(s,t\right);\Lambda_{t}}\\
\\
L\left(s,t\right)\myard{\textnormal{unitor}}{rr}\ar[uur]^{\Phi_{s,s,t}} &  & 1_{LX};L\left(s;t\right) &  & L\left(s,t\right)\myard{\textnormal{unitor}}{rr}\ar[uur]^{\Phi_{s,t,t}} &  & L\left(s,t\right);1_{LY}
}
\]
\end{enumerate}
\end{cor}

\begin{rem}
Note that this description of an oplax functor out of the bicategory
of spans does not involve pullbacks, thus allowing for a simpler for
a simpler proof of the universal properties of the span construction
\cite{unispans}. 
\end{rem}

\subsection{Bicategories of polynomials}

Given a locally cartesian closed category $\mathcal{E}$, one may
form the bicategory of polynomials in $\mathcal{E}$ with cartesian
2-cells \cite{weber,gambinokock}. This bicategory we denote by $\mathbf{Poly}_{c}\left(\mathcal{E}\right)$
and has objects those of $\mathcal{E}$, 1-cells given by diagrams
\[
\xymatrix@=1em{ & E\ar[ld]_{s}\ar[r]^{p} & B\ar[rd]^{t}\\
X &  &  & Z
}
\]
in $\mathcal{E}$ called polynomials and denoted by $\left(s,p,t\right)$,
and 2-cells given by commuting diagrams as below
\[
\xymatrix@=1em{ & K\ar[dd]_{f}\ar[ld]_{a}\ar[r]^{i}\ar@{}[ddr]|-{\textnormal{pb}} & I\ar[rd]^{b}\ar[dd]^{g}\\
X &  &  & Y\\
 & R\ar[lu]^{u}\ar[r]_{j} & J\ar[ur]_{v}
}
\]
where the middle square is a pullback. Composition of 1-cells is more
complicated and so will be omitted; especially as it is not necessary
to describe oplax functors out of $\mathbf{Poly}_{c}\left(\mathcal{E}\right)$
once we know the generics.

The reader need only know the following corollary of \cite[Prop. 3.1.6]{weber},
a description of polynomial composition due to Weber.
\begin{cor}
Consider two polynomials in $\mathcal{E}$ as below:
\[
\xymatrix@=1em{ & K\ar[ld]_{a}\ar[r]^{i} & I\ar[rd]^{b} &  &  &  &  & R\ar[r]^{j}\ar[ld]_{u} & J\ar[rd]^{v}\\
X &  &  & Y &  &  & Y &  &  & Z
}
\]
Then to give a cartesian 2-cell $\left(s,p,t\right)\to\left(a,i,b\right);\left(u,j,v\right)$
is to give a factorization $p=p_{1};p_{2}$ through an object $T$,
a morphism $h\colon T\to Y$, and a pair of cartesian morphisms $\left(s,p_{1},h\right)\to\left(a,i,b\right)$
and $\left(h,p_{2},t\right)\to\left(u,j,v\right)$ 
\[
\xymatrix@=1em{ &  & E\myar{p_{1}}{r}\ar@/^{1.5pc}/[rr]|-{p}\ar@{..>}[ldd]|-{w}\ar@/_{1pc}/[llddd]_{s}\ar@{}[dd]|-{\textnormal{pb}} & T\myar{p_{2}}{r}\ar@{..>}[rdd]|-{y}\ar@{..>}[ldd]|-{x}\ar[ddd]|-{h} & B\ar@{..>}[rdd]|-{z}\ar@/^{1pc}/[rrddd]^{t}\ar@{}[dd]|-{\textnormal{pb}}\\
\\
 & K\ar[ld]_{a}\ar[r]^{i} & I\ar[rd]^{b} &  & R\ar[r]^{j}\ar[ld]_{u} & J\ar[rd]^{v}\\
X &  &  & Y &  &  & Z
}
\]
such that the above diagram commutes. Here we identify a septuple
$\left(p_{1},h,p_{2},w,x,y,z\right)$ as above with another septuple
$\left(p_{1}',h',p_{2}',w',x',y',z'\right)$ if $w=w'$, $z=z'$ and
there exists an invertible $\alpha\colon T\to T'$ rendering commutative
the diagrams\footnote{It is clear that if the middle diagram commutes then the rightmost diagram also does. Also, such an isomorphism $\alpha$ making the left diagram commute must be unique.}
\begin{equation}
\xymatrix@=1em{ & T\ar[rd]^{p_{2}}\ar[dd]|-{\alpha} &  &  &  & T\ar[dd]|-{\alpha}\ar[rd]^{y}\ar[ld]_{x} &  &  & T\ar[dd]_{\alpha}\ar[rrd]^{h}\\
E\ar[ur]^{p_{1}}\ar[rd]_{p_{1}'} &  & B &  & I &  & R &  &  &  & Y\\
 & T'\ar[ur]_{p_{2}'} &  &  &  & T'\ar[ur]_{y'}\ar[ul]^{x'} &  &  & T'\ar[urr]_{h'}
}
\label{polytriple}
\end{equation}
\end{cor}

It follows that
\[
\mathbf{Poly}_{c}\left(\mathcal{E}\right)_{X,Z}\left(\left(s,p,t\right),\left(a,i,b\right);\left(u,j,v\right)\right)
\]
is isomorphic to
\begin{equation}
\sum_{p=p_{1};p_{2},\;h\colon T\to Y}^{\sim}\mathbf{Poly}_{c}\left(\mathcal{E}\right)_{X,Y}\left(\left(s,p_{1},h\right),\left(a,i,b\right)\right)\times\mathbf{Poly}_{c}\left(\mathcal{E}\right)_{Y,Z}\left(\left(h,p_{2},t\right),\left(u,j,v\right)\right)\label{polyfamrep}
\end{equation}
where the equivalence relation ``$\sim$'' indicates the sum is
taken over representatives of equivalence classes of triples $\left(p_{1},h,p_{2}\right)$
(where two such triples are seen as equivalent if there is an isomorphism
$\alpha$ rendering commutative the left and right diagrams as in
Figure \ref{polytriple}). We have thus exhibited the bicategory of
polynomials with cartesian 2-cells as a generic bicategory.

Here our class of generics $\Delta_{2}$ consists of, for each diagram
\[
\xymatrix@=1.5em{ &  & E\ar[ld]_{s}\ar[r]^{p_{1}} & T\ar[d]|-{h}\ar[r]^{p_{2}} & B\ar[rd]^{t}\\
 & X &  & Y &  & Z
}
\]
in $\mathcal{E}$ where $p=p_{1};p_{2}$, the cartesian morphisms
of polynomials 
\[
\delta_{s,p_{1},h,p_{2},t}\colon\left(s,p,t\right)\to\left(s,p_{1},h\right);\left(h,p_{2},t\right)
\]
 corresponding to
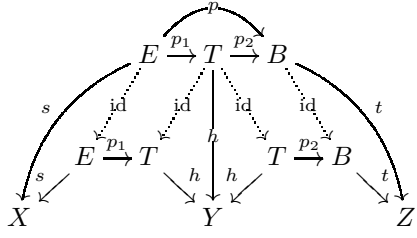
\[
\xymatrix@=1em{ &  & E\myar{p_{1}}{r}\ar@/^{1.5pc}/[rr]|-{p}\ar@{..>}|-{\textnormal{id}}[ldd]\ar@/_{1pc}/[llddd]_{s} & T\myar{p_{2}}{r}\ar@{..>}[rdd]|-{\textnormal{id}}\ar@{..>}[ldd]|-{\textnormal{id}}\ar[ddd]|-{h} & B\ar@{..>}[rdd]|-{\textnormal{id}}\ar@/^{1pc}/[rrddd]^{t}\\
\\
 & E\ar[ld]_{s}\ar[r]^{p_{1}} & T\ar[rd]^{h} &  & T\ar[r]^{p_{2}}\ar[ld]_{h} & B\ar[rd]^{t}\\
X &  &  & Y &  &  & Z
}
\]
under this bijection. We take as our augmentations the cartesian morphisms
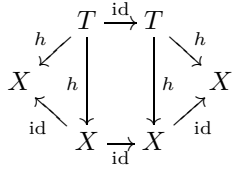
\[
\xymatrix@=1em{ & T\ar[ld]_{h}\ar[dd]_{h}\ar[r]^{\textnormal{id}} & T\ar[rd]^{h}\ar[dd]^{h}\\
X &  &  & X\\
 & X\ar[lu]^{\textnormal{id}}\ar[r]_{\textnormal{id}} & X\ar[ur]_{\textnormal{id}}
}
\]
and denote these by $\epsilon_{h}$. There are more general morphisms
into identity polynomials where the middle map is invertible; but
using those would lead to unnecessary complexity.
\begin{rem}
Note that our class of generics $\Delta_{2}$ does not involve representatives
of equivalence classes, unlike the summation formula given. 
\end{rem}

Now, applying Theorem \ref{generalcoherence} we have the following.
\begin{cor}
Let $\mathcal{E}$ be a locally cartesian closed category and denote
by $\mathbf{Poly}_{c}\left(\mathcal{E}\right)$ the bicategory of
polynomials in $\mathcal{E}$ with cartesian 2-cells. Let $\mathscr{C}$
be a bicategory. Then to give an oplax functor
\[
L\colon\mathbf{Poly}_{c}\left(\mathcal{E}\right)\to\mathscr{C}
\]
 is to give a locally defined functor
\[
L_{X,Y}\colon\mathbf{Poly}_{c}\left(\mathcal{E}\right)_{X,Y}\to\mathscr{C}_{LX,LY},\qquad X,Y\in\mathcal{E}
\]
with comultiplication and counit maps 
\[
\Phi_{s,p_{1},h,p_{2},t}\colon L\left(s,p,t\right)\to L\left(s,p_{1},h\right);L\left(h,p_{2},t\right),\qquad\Lambda_{h}\colon L\left(h,1,h\right)\to1_{LX}
\]
for every respective diagram in $\mathcal{E}$ 
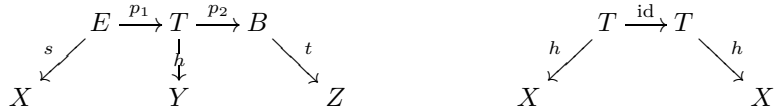
\[
\xymatrix@=1.5em{ &  & E\ar[ld]_{s}\ar[r]^{p_{1}} & T\ar[d]|-{h}\ar[r]^{p_{2}} & B\ar[rd]^{t} &  &  &  &  & T\ar[ld]_{h}\ar[r]^{\textnormal{id}} & T\ar[rd]^{h}\\
 & X &  & Y &  & Z &  &  & X &  &  & X
}
\]
where we assert $p=p_{1};p_{2}$ on the left, such that:
\begin{enumerate}
\item for any morphisms of polynomials as below
\[
\xymatrix@=1.5em{ & R\ar[dl]_{u}\ar[dd]|-{f}\ar[r]^{q} & I\ar[dr]^{v}\ar[dd]|-{g} &  &  & R\ar[dl]_{u}\ar[dd]|-{f}\ar[r]^{q_{1}} & S\ar[dr]^{k}\ar[dd]|-{c} &  &  & S\ar[dl]_{k}\ar[dd]|-{c}\ar[r]^{q_{2}} & I\ar[dr]^{v}\ar[dd]|-{g}\\
X &  &  & Z & X &  &  & Y & Y &  &  & Z\\
 & E\ar[ul]^{s}\ar[r]_{p} & B\ar[ru]_{t} &  &  & E\ar[ul]^{s}\ar[r]_{p_{1}} & T\ar[ru]_{h} &  &  & T\ar[ul]^{h}\ar[r]_{p_{2}} & B\ar[ru]_{t}
}
\]
we have the commuting diagram
\[
\xymatrix{L\left(u,q,v\right)\ar[d]_{L\left(f,g\right)}\myar{\Phi_{u,q_{1},k,q_{2},v}}{rr} &  & L\left(u,q_{1},k\right);L\left(k,q_{2},v\right)\ar[d]^{L\left(f,c\right);L\left(c,g\right)}\\
L\left(s,p,t\right)\myard{\Phi_{s,p_{1},h,p_{2},t}}{rr} &  & L\left(s,p_{1},h\right);L\left(h,p_{2},t\right)
}
\]
\item for any morphism of polynomials as on the left below
\[
\xymatrix@=1.5em{ & R\ar[dl]_{s}\ar[dd]_{f}\ar[r]^{\textnormal{id}}\ar@{}[rdd]|-{\textnormal{pb}} & I\ar[dr]^{s}\ar[dd]^{f} &  &  & L\left(s,1,s\right)\ar[rr]^{L\left(f,f\right)}\ar[rdd]_{\Lambda_{s}} &  & L\left(t,1,t\right)\ar[ldd]^{\Lambda_{t}}\\
X &  &  & Z\\
 & T\ar[ul]^{t}\ar[r]_{\textnormal{id}} & J\ar[ru]_{t} &  &  &  & 1_{LX}
}
\]
the diagram on the right above commutes;
\item for all diagrams of the form
\[
\xymatrix@=1.5em{ & F\ar[dl]_{s}\ar[r]^{a} & G\ar[r]^{b}\ar[d]_{h} & H\ar[r]^{c}\ar[d]^{k} & K\ar[rd]^{t}\\
W &  & X & Y &  & Z
}
\]
in $\mathcal{E}$, we have the commuting diagram
\[
\xymatrix{L\left(s,a;b;c,t\right)\ar[d]_{\Phi_{s,a,h,b;c,t}}\ar@{=}[rr] &  & L\left(s,a;b;c,t\right)\ar[d]^{\Phi_{s,a;b,k,c,t}}\\
L\left(s,a,h\right);L\left(h,b;c,t\right)\myard{L\left(s,a,h\right);\Phi_{h,b,k,c,t}}{d} &  & L\left(s,a;b,k\right);L\left(k,c,t\right)\ar[d]^{\Phi_{s,a,h,b,k};L\left(k,c,t\right)}\\
L\left(s,a,h\right);\left(L\left(h,b,k\right);L\left(k,c,t\right)\right)\myard{\textnormal{assoc}}{rr} &  & \left(L\left(s,a,h\right);L\left(h,b,k\right)\right);L\left(k,c,t\right)
}
\]
\item for all polynomials $\left(s,p,t\right)$ the diagrams
\[
\xymatrix@=0.5em{ & L\left(s,1,s\right);L\left(s,p,t\right)\ar[rdd]^{\Lambda_{s};L\left(s,p,t\right)}\\
\\
L\left(s,p,t\right)\myard{\textnormal{unitor}}{rr}\ar[uur]^{\Phi_{s,1,s,p,t}} &  & 1_{LX};L\left(s,p,t\right)\\
\\
 & L\left(s,p,t\right);L\left(t,1,t\right)\ar[rdd]^{L\left(s,p,t\right);\Lambda_{t}}\\
\\
L\left(s,p,t\right)\myard{\textnormal{unitor}}{rr}\ar[uur]^{\Phi_{s,p,t,1,t}} &  & L\left(s,p,t\right);1_{LY}
}
\]
commute.
\end{enumerate}
\end{cor}

\begin{rem}
As the above description of oplax functors out of the bicategory of
polynomials does not rely on polynomial composition, it may be used
for an efficient proof of the universal properties of polynomials.
Indeed, this allows us to avoid the large coherence diagrams which
would arise in a direct proof. We will discuss this in detail in our
next paper.
\end{rem}

\subsection{Finite sets and bijections}

We give this example for completeness, but will omit some details
as Theorem \ref{generalcoherence} becomes rather trivial in this
case (due to all generic morphisms being invertible). Here we take
$\mathscr{A}$ to be the category of finite sets and bijections with
the disjoint union monoidal structure, denoted $\left(\mathbb{P},\sqcup,\emptyset\right)$.
This monoidal category is generic since we have isomorphisms
\[
\mathbb{P}\left(C,A\sqcup B\right)\cong\sum_{C=L\sqcup R}\mathbb{P}\left(L,A\right)\times\mathbb{P}\left(R,B\right)
\]
natural in finite sets $A$ and $B$, where the sum is taken over
decompositions of $C$ into the disjoint union of two sets. Here we
choose our class of generics $\Delta_{2}$ to contain the chosen bijection,
where $\left[n\right]=\left\{ 1,\cdots,n\right\} $, 
\[
\delta{}_{n_{1},n_{2}}\colon\left[n_{1}+n_{2}\right]\to\left[n_{1}\right]\sqcup\left[n_{2}\right],\qquad n\mapsto\begin{cases}
\left(1,n\right), & n\leq n_{1}\\
\left(2,n\right), & n>n_{1}
\end{cases}
\]
for each pair of non-negative integers $n_{1}$ and $n_{2}$. Trivially,
the only augmentation is the identity map on the empty set. Taking
$\left(\mathcal{C},\otimes,I\right)$ to be a monoidal category, it
follows from Theorem \ref{generalcoherence} that oplax monoidal functors
$L\colon\left(\mathbb{P},\sqcup,\emptyset\right)\to\left(\mathcal{C},\otimes,I\right)$
may be specified by giving comultiplication and counit maps
\[
\Phi_{n_{1},n_{2}}\colon L\left[n_{1}+n_{2}\right]\to L\left[n_{1}\right]\otimes\left[n_{2}\right],\qquad\Lambda\colon L\left(\emptyset\right)\to I
\]
Of course, this may more easily be seen by simply taking the skeleton.

\section{Convolution structures and Yoneda structures\label{GenericDocYoneda}}

By results of Day \cite{dayconvolution}, given a bicategory $\mathscr{A}$
with locally small hom-categories one may consider the local cocompletion
of $\mathscr{A}$, a new bicategory $\hat{\text{\ensuremath{\mathscr{A}}}}$
with objects those of $\mathscr{A}$, hom-categories given by
\[
\hat{\text{\ensuremath{\mathscr{A}}}}_{X,Y}:=\left[\mathscr{A}_{X,Y}^{\textnormal{op}},\mathbf{Set}\right],\qquad X,Y\in\mathscr{A}_{\textnormal{ob}}
\]
and a composite of two presheaves 
\[
F\colon\mathscr{A}_{X,Y}^{\textnormal{op}}\to\mathbf{Set},\qquad G\colon\mathscr{A}_{Y,Z}^{\textnormal{op}}\to\mathbf{Set}
\]
given by Day's convolution formula
\[
GF\colon\mathscr{A}_{X,Z}^{\textnormal{op}}\to\mathbf{Set},\qquad GF\left(c\right)=\int^{a,b}\mathscr{A}_{X,Z}\left(c,a;b\right)\times Fa\times Gb
\]
With this definition, the family of Yoneda embeddings on the hom-categories
defines a pseudofunctor $y_{\mathscr{A}}\colon\mathscr{A}\to\hat{\mathscr{A}}$.
This is of interest since in the case of generic bicategories $\mathscr{A}$,
this convolution structure has an especially nice form.
\begin{prop}
\label{generalconvolution} Suppose $\mathscr{A}$ is a generic bicategory.
Then for any pair of presheaves
\[
F\colon\mathscr{A}_{X,Y}^{\textnormal{op}}\to\mathbf{Set},\qquad G\colon\mathscr{A}_{Y,Z}^{\textnormal{op}}\to\mathbf{Set}
\]
there exists isomorphisms as below
\[
\int^{a,b}\mathscr{A}_{X,Z}\left(c,a;b\right)\times Fa\times Gb\cong\sum_{m\in\mathfrak{M}_{c}^{X,Y,Z}}Fl_{m}\times Gr_{m}
\]
thus reducing the Day convolution structure to a simpler formula.
\end{prop}

\begin{proof}
We have
\[
\begin{aligned}\textnormal{LHS} & =\int^{a,b}\mathscr{A}_{X,Z}\left(c,a;b\right)\times Fa\times Gb\\
 & \cong\int^{a,b}\left[\sum_{m\in\mathfrak{M}_{c}^{X,Y,Z}}\mathscr{A}_{X,Y}\left(l_{m},a\right)\times\mathscr{A}_{Y,Z}\left(r_{m},b\right)\right]\times Fa\times Gb\\
 & \cong\int^{a,b}\sum_{m\in\mathfrak{M}_{c}^{X,Y,Z}}\mathscr{A}_{X,Y}\left(l_{m},a\right)\times Fa\times\mathscr{A}_{Y,Z}\left(r_{m},b\right)\times Gb\\
 & \cong\sum_{m\in\mathfrak{M}_{c}^{X,Y,Z}}\int^{a,b}\mathscr{A}_{X,Y}\left(l_{m},a\right)\times Fa\times\mathscr{A}_{Y,Z}\left(r_{m},b\right)\times Gb\\
 & \cong\sum_{m\in\mathfrak{M}_{c}^{X,Y,Z}}\left(\int^{a}\mathscr{A}_{X,Y}\left(l_{m},a\right)\times Fa\right)\times\left(\int^{b}\mathscr{A}_{Y,Z}\left(r_{m},b\right)\times Gb\right)\\
 & \cong\sum_{m\in\mathfrak{M}_{c}^{X,Y,Z}}Fl_{m}\times Gr_{m}\\
 & =\textnormal{RHS}
\end{aligned}
\]
as required.
\end{proof}
\begin{rem}
Unfortunately, the above formula has some disadvantages. Indeed, as
$\mathfrak{M}_{c}^{X,Y,Z}$ is isomorphic to the set of equivalence
classes of generics out of $c$, it follows that writing down $\mathfrak{M}_{c}^{X,Y,Z}$
will involve a choice of representatives for each equivalence class.
This is problematic since choices of representatives do not nicely
behave with respect to composition.
\end{rem}

As a consequence of this proposition and the formulas \eqref{spanfamrep}
and \eqref{polyfamrep} given in the previous section, we have the
following.
\begin{cor}
The Day convolution of two presheaves of spans
\[
F\colon\mathbf{Span}\left(\mathcal{E}\right)_{X,Y}^{\textnormal{op}}\to\mathbf{Set},\qquad G\colon\mathbf{Span}\left(\mathcal{E}\right)_{Y,Z}^{\textnormal{op}}\to\mathbf{Set}
\]
is given by 
\[
GF\colon\mathbf{Span}\left(\mathcal{E}\right)_{X,Z}^{\textnormal{op}}\to\mathbf{Set},\quad GF\left(s,t\right)\cong\sum_{h\colon T\to Y}F\left(s,h\right)\times G\left(h,t\right)
\]
and the Day convolution of two presheaves of polynomials
\[
F\colon\mathbf{Poly}_{c}\left(\mathcal{E}\right)_{X,Y}^{\textnormal{op}}\to\mathbf{Set},\qquad G\colon\mathbf{Poly}_{c}\left(\mathcal{E}\right)_{Y,Z}^{\textnormal{op}}\to\mathbf{Set}
\]
is given by the formula
\[
GF\colon\mathbf{Poly}_{c}\left(\mathcal{E}\right)_{X,Z}^{\textnormal{op}}\to\mathbf{Set},\quad GF\left(s,p,t\right)\cong\sum_{p=p_{1};p_{2},\;h\colon T\to Y}^{\sim}F\left(s,p_{1},h\right)\times G\left(h,p_{2},t\right)
\]
\end{cor}

The purpose of the following is to describe how Theorem \ref{generalcoherence}
may be seen as an instance of a more general result. Indeed, as a
special case of \cite[Theorem 76]{WalkerDL} we have the following
corollary.
\begin{cor}
[Doctrinal Yoneda Structures] Let $\mathscr{A}$ and $\mathscr{C}$
be bicategories with locally small hom-categories. Let $\hat{\mathscr{A}}$
be the free small local cocompletion of $\mathscr{A}$. Then for any
locally defined identity on objects functor $L\colon\mathscr{A}\to\mathscr{C}$,
with the corresponding locally defined identity on objects functor
$R=\mathscr{C}\left(L-,-\right)$ as below
\[
\xymatrix@=1em{\mathscr{C}\myar{R}{rr} &  & \mathcal{\hat{\mathscr{A}}}\ar@{}[ld]|-{\stackrel{\varphi}{\Longleftarrow}}\\
 & \;\\
 &  & \mathscr{A}\ar[uu]_{y_{\mathscr{A}}}\ar[uull]^{L}
}
\]
the structure of an oplax functor on $L$ is in bijection with the
structure of a lax functor on $R$.
\end{cor}

Supposing that $\mathscr{A}$ is generic, and hence that composition
on $\hat{\mathscr{A}}$ has the reduced form given by Proposition
\ref{generalconvolution}, one sees that for a given locally defined
functor $L\colon\mathscr{A}\to\mathscr{C}$, giving an oplax functor
$\left(L,\varphi,\lambda\right)\colon\mathscr{A}\to\mathscr{C}$ with
constraint cells 
\[
\varphi_{a,b}\colon L\left(a;b\right)\to La;Lb,\qquad\lambda_{X}\colon L1_{X}\to1_{X}
\]
is to give a lax functor $\left(R,\phi,\omega\right)\colon\mathscr{C}\to\hat{\mathscr{A}}$
with constraints 
\[
\phi_{a,b}\colon Ra;Rb\to R\left(a;b\right),\qquad\omega_{X}\colon1_{X}\to R1_{X}
\]
These binary constraints are functions for each $c\colon X\to Z$
\[
\sum_{m\in\mathfrak{M}_{c}^{X,Y,Z}}\mathscr{C}_{X,Y}\left(Ll_{m},a\right)\times\mathscr{C}_{Y,Z}\left(Lr_{m},b\right)\to\mathscr{C}_{X,Z}\left(Lc,a;b\right)
\]
natural in $a,b$ and $c$. By naturality, to give such a function
is to give an assignment on the identity pair (we may call the result
$\Phi_{c,m}$)
\[
\left(\textnormal{id}\colon Ll_{m}\to Ll_{m},\textnormal{id}\colon Lr_{m}\to Lr_{m}\right)\mapsto\Phi_{c,m}\colon Lc\to Ll_{m};Lr_{m}
\]
A similar calculation may be done with the nullary constraints $\Lambda$.
\begin{rem}
It is this observation which is the motivation for Theorem \ref{generalcoherence}.
However, this approach does not give an efficient proof of this theorem
for a number of technical reasons. In particular, we wish to avoid
considering equivalence classes of generic morphisms (such as the
set $\mathfrak{M}_{c}^{X,Y,Z}$) to avoid technicalities involving
choices of representatives.
\end{rem}

\bibliographystyle{siam}
\bibliography{references}

\end{document}